\documentclass[a4paper,10pt,leqno,english]{smfart}
\usepackage{aeguill}
\usepackage{enumerate}
\usepackage{amssymb,amsmath,latexsym,amsthm}
\usepackage[T1]{fontenc}
\usepackage{smfthm}      
\usepackage{geometry}   
\usepackage{url}   
\usepackage[frenchb, english]{babel}
\usepackage[utf8]{inputenc}   
\usepackage{mathrsfs}
\usepackage{xcolor}
\usepackage{comment}
\definecolor{violet}{rgb}{0.0,0.2,0.7}
\definecolor{rouge2}{rgb}{0.8,0.0,0.2}
\usepackage{hyperref}
\usepackage{relsize}

\hypersetup{
    bookmarks=true,         
    unicode=false,          
    pdftoolbar=true,        
    pdfmenubar=true,        
    pdffitwindow=false,     
    pdfstartview={FitH},    
    pdftitle={},    
    pdfauthor={},     
    colorlinks=true,       
   linkcolor=violet,          
    citecolor=rouge2,        
    filecolor=black,      
    urlcolor=cyan}           
\numberwithin{equation}{section}

\newcommand{\R}{\mathbb{R}}
\newcommand{\CC}{\mathbb{C}}
\newcommand{\Q}{\mathbb{Q}}
\newcommand{\Z}{\mathbb{Z}}

\newcommand{\bp}{\mathbb{B}_+}

\renewcommand{\d}{\partial}

\newcommand{\vp}{\varphi}

\renewcommand{\O}{\mathcal{O}}
\newcommand{\Ox}{\mathcal{O}_{X}}
\newcommand{\ep}{\varepsilon}
\renewcommand{\epsilon}{\varepsilon}

\newcommand{\la}{\langle}
\newcommand{\al}{\alpha}
\newcommand{\ra}{\rangle}

\renewcommand{\ge}{\geqslant}
\renewcommand{\le}{\leqslant}
\renewcommand{\leq}{\leqslant}
\renewcommand{\geq}{\geqslant}
\newcommand{\Ric}{\mathrm{Ric} \,}
\newcommand{\om}{\omega}
\newcommand{\omtlc}{\om_{t,lc}}
\newcommand{\omke}{\omega_{\rm KE}}
\newcommand{\ddc}{dd^c}
\newcommand{\xreg}{X_{\rm reg}}
\newcommand{\xsing}{X_{\rm sing}}

\newcommand{\vpte}{\varphi_{t,\ep}}
\newcommand{\pste}{\psi_{t,\ep}}
\newcommand{\ute}{u_{t,\varepsilon}}

\newcommand{\vpe}{\varphi_{\varepsilon}}

\newcommand{\D}{D}
\newcommand{\Dlc}{D_{lc}}
\newcommand{\Dklt}{D_{klt}}
\newcommand{\Xlc}{X_{lc}}

\newcommand{\cka}{\mathscr{C}^{k,\alpha}_{qc}}

\newcommand{\omt}{\om_{t}}
\newcommand{\omc}{\om_{\chi}}

\newcommand{\Supp}{\mathrm {Supp}}
\newcommand{\tr}{\mathrm{tr}}

\newcommand{\amp}{\mathrm{Amp}(\alpha)}
\newcommand{\exc}{\mathrm{Exc}}
\newcommand{\vol}{\mathrm{vol}}
\newcommand{\MA}{\mathrm{MA}}

\newcommand{\Vt}{V_{\theta}}
\newcommand{\capt}{\mathrm{Cap}_{\theta}}
\newcommand{\vpd}{\vp:X\to \R\cup\{-\infty\}}
\newcommand{\Xnu}{X^{\nu}}
\newcommand{\cond}{\textrm{cond}}
\newcommand{\E}{\mathcal{E}}
\newcommand{\Ej}{\mathcal{E}_{\omega_{j}}}
\newcommand{\G}{\mathcal{G}}
\renewcommand{\L}{\mathcal{L}}
\newcommand{\vde}{v_{\delta, \varepsilon}}
\newcommand{\vd}{v_{\delta}}

\newtheorem*{thma}{Theorem A}
\newtheorem*{thmb}{Theorem B}
\newtheorem*{thmc}{Theorem C}

\begin{document}

\title[KE metrics on stable varieties and lc pairs]{Kähler-Einstein metrics on stable varieties
and log canonical pairs}

\date{\today}
\author{Robert J. Berman}
\address{Chalmers Techniska Hogskola, Göteborg, Sweden}
\email{robertb@math.chalmers.se}

\author{Henri Guenancia}

\address{Institut de Mathématiques de Jussieu \\
Université Pierre et Marie Curie \\
 Paris 
\& Département de Mathématiques et Applications \\
\'Ecole Normale Supérieure \\
Paris}
\email{guenancia@math.jussieu.fr}


\maketitle
\tableofcontents

\section*{Introduction}

According to the seminal works of Aubin \cite{Aubin} and Yau \cite{Yau78} any \emph{canonically
polarized }compact complex manifold $X$ (i.e. $X$ is a non-singular
projective algebraic variety such that the canonical line bundle $K_{X}$
is ample) admits a unique Kähler-Einstein metric $\omega$ in the
first Chern class $c_{1}(K_{X}).$ One of the main goals of the present
paper is to extend this result to the case when $X$ is singular or
more precisely when $X$ has \emph{semi-log canonical} singularities. A major
motivation comes from the fact that such singular varieties are used to compactify the moduli space of canonically polarized
manifolds - a subject where there has been great progress in the last
years in connection to the (log) Minimal Model Program (MMP) in birational
algebraic geometry \cite{Kollar, Kov}. The varieties in question are usually
refered to as \emph{stable} \emph{varieties} (or \emph{canonical models})
as they are the higher dimensional generalization of the classical
notion of stable curves of genus $g>1$, which form the Deligne-Mumford
compactification of the moduli space of non-singular genus $g$ curves
\cite{Kollar, Kov}. It as a classical fact that any stable curve admits
a unique Kähler-Einstein metric on its regular part, whose total area
is equal to the (arithmetic) degree of the curve $X$ (see the section on stable curves further in the introduction for more details). Our first (and main) result gives a generalization of this fact to the higher dimensional
setting:

\begin{thma}
Let $X$ be a projective complex algebraic variety with semi-log canonical
singularities such that $K_{X}$ is ample. Then there exists a Kähler metric $\om$ on the regular locus $\xreg$, satisfying 
\[\Ric \om = -\om\]
 and such that the volume of $(\xreg,\omega)$ coincides with
the volume of $K_{X},$ i.e. $\int_{\xreg}\omega^{n}=c_{1}(K_{X})^{n}.$ 
Moreover, the metric extends to define a current $\omega$ in $c_{1}(K_{X})$ which
is uniquely determined by X. 
\end{thma}

We will refer to the current $\omega$ in the previous theorem as
a (singular) Kähler-Einstein metric on $X.$ Moreover, the current
$\omega$ will be shown to be of\emph{ finite energy,} in the sense
of \cite{GZ07, BBGZ} and as discussed in the last section of the present paper this allows one to define a canonical (singular)
Weil-Peterson metric on the compact moduli space in terms of Deligne pairings.  
The notion of semi-log canonical singularities of a variety
$X$ - which is the most general class of singularities appearing in the (log) Minimal Model Program - will be recalled below. For the moment let us just point out
that the definition involves two ingredients: first a condition which
makes sure that the canonical divisor $K_{X}$ is defined as a $\Q-$
Cartier divisor (i.e. $\Q-$line bundle) which is in particular needed
to make so sense of the notion of ampleness of $K_{X}$ and secondly,
the definition of semi-log canonical singularities involves a bound on the discrepancies of $X$ on any resolution of singularities. \\

In fact, we will conversely show that if $K_{X}$ is ample and the
variety $X$ admits a Kähler-Einstein metric then $X$ has semi-log canonical
singularities and this brings us to our second motivation for studying
Kähler-Einstein metrics in the the singular setting, namely the Yau-Tian-Donaldson
conjecture. Recall that this conjecture concerns polarized algebraic manifolds
$(X,L),$ i.e. algebraic manifolds together with an ample line bundle
$L\rightarrow X$ and it says that the first Chern class $c_{1}(L)$
of $L$ contains a Kähler metric $\omega$ with constant scalar curvature
if and only if $(X,L)$ is $K$-stable. The latter notion of stability
is of an algebro-geometric nature and can be seen as an asymptotic
form of the classical notions of Chow and Hilbert stability appearing
in Geometric Invariant Theory (GIT). However, while the notion of
$K$-stability makes equal sense when $X$ is singular it is less clear
how to give a proper definition of a constant scalar curvature metric
for a singular polarized variety $(X,L).$ But, as it turns out,
the situation becomes more transparent in the case when $L$ is equal
to $K_{X}$ or its dual, the anti-canonical bundle $-K_{X}.$ The starting
point is the basic fact that, when $X$ is smooth, a Kähler metric
in $\omega$ in $c_{1}(\pm K_{X})$ has constant scalar curvature
on all of $X$ precisely when it has constant Ricci curvature, i.e.
when $\omega$ is a Kähler-Einstein metric. Various generalizations
of Kähler-Einstein metrics to the singular setting have been proposed
in the litterature, see e.g. \cite{EGZ, BEGZ, BBEGZ} etc.
In this paper we will adopt the definition which appears in the formulation
of the previous theorem (see also section \ref{sec:singke}), i.e. a positive current in $c_{1}(\pm K_{X})$
is said to define a (singular) Kähler-Einstein metric if defines a
bona fide Kähler-Einstein metric on the regular locus $\xreg$ and
if its total volume there coincides with the algebraic top intersection
number of $c_{1}(\pm K_{X}).$ This definition, first used in the
Fano case in \cite{BBEGZ}, has the virtue of generalizing all previously proposed
definitions, regardless of the sign of the canonical line bundle. Combing
our results with recent results of Odaka \cite{Odaka, od2}, which say that a
canonicaly polarized variety has semi-log canonical singularities
precisely when $(X,K_{X})$ is $K$-stable, gives the following theorem,
which can be seen as a confirmation of the generalized form of the
Yau-Tian-Donaldson conjecture for canonically polarized varieties (satisfying the conditions $G_1$ and $S_2$, cf \ref{thm:sumup} for a more precise statement):

\begin{thmb}
Let $X$ be a complex projective variety such that $K_{X}$
is ample. Then $X$ admits a Kähler-Einstein metric if and only if
$(X,K_{X})$ is $K-$stable.
\end{thmb}

It may also be illuminating to compare this result with the case
when $L:=-K_{X}$ is ample (i.e. $X$ is Fano). Then it was shown
in \cite{Ber}, in the general singular setting, that the existence of a Kähler-Einstein
metrics indeed implies $K$-(poly)stability. As for the converse it was
finally settled very recently in the deep works by Chen-Donaldson-Sun
\cite{CDS1,CDS2,CDS3} and Tian \cite{T}, independently, in the case when $X$ is smooth.
The existence problem in the singular case is still open in general,
except for the toric case \cite{BB}; cf also \cite{OSS} for a related problem in the case of singular Fano surfaces.\\

Coming back to the present setting we point out that the starting
point of our approach is that, after passing to a suitable resolution
of singularities, we may as well assume that the variety $X$ is smooth
if we work in the setting of \emph{log pairs} $(X,D),$ where $D$
is a $\Q-$divisor on $X$ with simple normal crossings (SNC) and
where the role of the canonical line bundle is played by the 
\emph{log canonical line bundle} $K_{X}+D$ (which appears as the pull-back
to the resolution of the original canonical line bundle). In this
notation the original variety has semi-log canonical singularities
precisely when the log pair $(X,D)$ is \emph{log canonical (lc)} in
the usual sense of the Minimal Model Program, i.e. the coefficents
of $D$ are at most equal to one (but negative coefficents are allowed).
However, it should be stressed that for this gain in regularity we
have, of course, to pay a loss of positivity: even if the original
canonical line bundle is ample, the corresponding log canonical line
bundle is only \emph{semi-ample} (and big) on the resolution, since
it is trivial along the exceptional divisors of the corresponding
resolution. 

The upshot is that the natural setting for our results is the setting
of log smooth log canonical pairs $(X,D)$ such that the log canonical line bundle
$K_{X}+D$ is semi-ample and big. To any such pair we will associate
a canonical Kähler-Einstein metric $\omega$ in the sense that $\omega$
is a current in the first Chern class $c_{1}(K_{X}+D)$ such that
$\omega$ restricts to a bona fide Kähler-Einstein metric on a Zariski
open set of $X$ and such that, globally on $X$, the current defined
by the divisor $D$ gives a singular contribution to the Ricci curvature
of $\omega.$ \\

\begin{thmc}
Let $X$ be a Kähler manifold and $D$ a simple normal crossings $\R-$divisor
on $X$ with coefficients in $]-\infty,1]$ such that $K_{X}+D$ is
semi-positive and big (i.e. $(K_{X}+D)^{n}>0).$ Then there exists
a unique current $\omega$ in $c_{1}(K_{X}+D)$ which is smooth on
a Zariski open set $U$ of $X$ and such that 
\[\Ric \om=-\om+[D]\]
 holds on $X$ in the weak sense and $\int_{U}\omega^{n}=(K_{X}+D)^{n}.$
More precisely, $U$ can be taken to be the complement of $D$ in
the ample locus of $K_{X}+D.$ Moreover, any such current $\omega$ on $X$ automatically has finite energy.

\end{thmc}

Recall that the \emph{ample locus} of a big line bundle $L$ may be
defined as the Zariski open set whose complement is the \emph{augmented
base locus} of $L$, i.e. the intersection of all effective $\Q-$divisors
$E$ such that $L-E$ is ample. In particular, if $Y$ is a
variety with semi-log canonical singularities and $\pi:(X,D)\to Y$ is a log resolution
of the normalization (endowed with its conductor), then the exceptional locus of $\pi$ is contained in the augmented
base locus of $K_{X}+D$ and hence Theorem C above indeed implies Theorem A. \\

The existence proof of Theorem C (and its generalizations described
below) will be divided into two parts:
in the first part we construct a \emph{variational solution} with
\emph{finite energy}, by adapting the variational techniques developed
in \cite{BBGZ} to the present setting. Then, in the second part, we show that
the variational solutions have appropriate regularity using a priori Laplacian estimates, building
on the works of Aubin \cite{Aubin} and Yau \cite{Yau78} and ramifications of their work
by Kobayashi \cite{KobR} and Tian-Yau \cite{Tia} to the setting of quasi-projective
varieties \--- in particular we will be relying on Yau's maximum principle.
For the second part we will need to perturb the line bundle $L:=K_{X}+D$
(to make it ample) and regularize the klt part of the divisor $D$
(to make the divisor purely log canonical). 

It is interesting to note that, so far, we are not able of proving Theorem C without using the variational method, i.e. relying only on \textit{a priori} estimates. The reason is that our estimates on the potential of the solution near $\Supp(D)$ are not good enough to extend it directly as a current with full Monge-Ampère mass, so as to get \textit{the} Kähler-Einstein metric of $(X,D)$.

Let us also point out that the variational part of our proof only requires $K_{X}+D$
to be \emph{big} and thus produces a unique singular Kähler-Einstein
metric of finite energy on any variety of log general type. As for
the regularity part it applies as long as the corresponding log canonical
ring is \emph{finitely generated}. In fact, according to one of the
fundamental conjectures of the general log minimal model program the
latter finiteness property always holds. Note that for log canonical
pairs, the finite generation is known to hold for $n\leq4$, cf \cite{Kollar92,Fujino} and the references therein. We also recall
that in the case of varieties of log general type with log terminal
singularities (which in our notation means that $D_{lc}$ vanishes)
the finite generation in question was established in the seminal work \cite{BCHM} 
and Kähler-Einstein metrics were first constructed in \cite{EGZ}.\\

In the last section of the paper some applications of Theorem A are
given. First, we explain the link with Yau-Tian-Donaldson as we indicated above in Theorem B. Then,
 we give a short analytic proof of the fact that the automorphism
group of a stable variety is finite (see \cite{BHPS} for algebro-geometric proofs). We also
discuss the problem of deducing Miyaoka-Yau type inequalities from Theorem
A. 

\subsection*{Further comparison with previous results}
\subsubsection*{Stable curves}

A stable curve, as defined by \cite{DM}, is a reduced one equidimensional projective scheme over $\CC$ with only nodes as singularities and with finite automorphism group. The latter finiteness assumption can be replaced with various equivalent conditions, for example that the canonical sheaf of $X$ is ample or in differential geometric terms: every connected component of $\xreg = X \setminus \{\mathrm{nodes}\}$ is covered by the disk. In turn, this is equivalent to asking that every connected component of $\xreg$ admits a \textit{complete} hyperbolic metric. 

As  above, the higher dimensional analogue of Deligne-Mumford stable curves are the so-called stable varieties which are reduced equidimensional complex projective schemes with semi-log canonical singularities (i.e. double normal crossing singularities in codimension one, and log canonical singularities in higher codimension) and ample canonical bundle. In the light of the discussion above one could be tempted to believe that the regular locus of a stable variety can always be endowed with a complete Kähler-Einstein metric. However, this is not the case, mainly because of the singularities in codimension $\ge 2$ (the relationship between the completeness of the metric and the singularities of the variety will be analyzed in \cite{GW}). Indeed, our main result says that if we are given a scheme $X$ (reduced, equidimensional, complex and projective) whose only singularities in codimension one are double normal crossings and such that $K_X$ is ample, then $X$ has semi-log canonical singularities (i.e. $X$ is stable) if and only if $\xreg$ admits a Kähler-Einstein metric $\om$ with negative curvature such that $\mathrm{Vol}_{\om}(\xreg)= c_1(K_X)^n$. This volume condition replaces in higher dimension the completeness condition, which does not hold in general. Further, this condition was known for a long time to be equivalent to the other ones in the one dimensional case already, see e.g. \cite[\S 1]{HK}.

\subsubsection*{Quasi-projective varieties}
Theorem C also extends some of the results of Wu in \cite{Wu, Wu2},
concerning the setting of Kähler-Einstein metrics on quasi-projective
projective varieties of the form $X_{0}:=X-D,$ where $X$ is smooth
and $D$ is reduced SNC divisor. We recall that the case when $K_{X}+D$
is ample was independently settled by Kobayashi \cite{KobR} and Tian-Yau \cite{Tia}.
The case when $X$ is an orbifold and $K_{X}+D$ is semi-ample
and big was considered by Tian-Yau in \cite{Tia} and as later shown by Yau \cite{Yau3}
the corresponding Kähler-Einstein metric is then complete on $X_{0}.$  (in the orbifold sense). However, in our
general setting the metric will typically not be complete on the regular locus. This is only partly due to the klt singularities
(which generalize orbifold singularities) \--- there is also a complication coming from the presence of negative coefficents on a resolution.

To illustrate this we recall that a standard example of log canonical pairs $(X,D)$ is given by the
Borel-Baily compactification $X:=X_{0}-D$ of an arithmetic quotient,
i.e. $X_{0}=B/\Gamma,$ where $B$ is a bounded symmetric domain and
$\Gamma$ is discrete subgroup of the automorphism group of $B.$
In this case any toroidal resolution $X'$ has the property that the
corresponding divisor $D'$ on the resolution is reduced (and hence purely log
canonical) if $\Gamma$ is neat, i.e. if there are no fixed points.
The corresponding Kähler-Einstein metric on $X_{0}$ is the complete
one induced from the corresponding metric on $B,$ constructed in \cite{CY, MY}. When $\Gamma$ has fixed points these give rise
to an additional fractional klt part $D_{klt}'$ in $D'$ so that
the corresponding Kähler-Einstein metric is only complete in the orbifold
sense \cite{Tia}. However, for general log canonical singularity $(X,D)$
the klt part $D'_{klt}$ of $D'$ may not be fractional or more seriously:
it may contain negative coefficients and the main novelty of the present
paper is to show how to deal with this problem by combining a variational
approach with a priori estimates.

\subsubsection*{Behaviour at the boundary}
It is also interesting to compare with the case when the pair $(X,D)$
is log smooth with $K_{X}+D$ ample and with $D$ effective and klt
(i.e. with coefficients in $[0,1[),$ where very precise regularity
results have been obtained recently. For example, in \cite{Brendle, CGP, JMR, GP} it is
shown that the corresponding Kähler-Einstein metric $\omega$ has
conical singularities along $D$ (sometimes also called edge singularities in the litterature), thus confirming a previous conjecture
of Tian. As for the mixed case when the coefficient 1 is also allowed
in $D$ it was studied in \cite{G12, GP}, where it was shown that $\omega$
has mixed cone and Poincaré type singularities. A commun theme in
these results is that singularities of the metric $\omega$ are encoded
by a suitable local model (with cone or Poincaré type singularities)
determined by $D.$ However, the difficulty in the situation studied
in the present paper is the presence of negative coefficients in $D$
and the associated loss of positivity which appears when we pass to
a log resolution of a singular variety $X.$ It would be very interesting
if one could associate local models to this situation as well, but
this seems very challenging even in the case when $X$ has canonical singularities.

We should mention that the behaviour of the Kähler-Einstein metric of a log canonical pair $(X,D)$ such that $K_X+D$ is ample and $D$ is effective will be investigated in \cite{GW}. As a consequence of the results therein, the Kähler-Einstein metric of a stable variety is equivalent to the cusp metric near the ordinary double points.

\subsection*{Organization of the paper}
\begin{itemize}
\item[$\bullet$] \S \ref{sec:prelim}: We introduce the preliminary material that we will need,
concerning the pluripotential theoretic setting of singular metrics
on line bundles over varieties which are not necessarily normal. 
\item[$\bullet$] \S \ref{sec:singke}: Here we give the definition of a Kähler-Einstein metric
on a canonically polarized variety $X$ and more generally on a log
pair $(X,D).$ As we explain a purely differential-geometric definition
can be given which only involves the regular locus $\xreg$ of $X.$
But, at we show, the corresponding metric automatically extends in
a unique manner to define a singular current on $X$ (which will allow
us to prove the uniqueness of the Kähler-Einstein metric, later on
in section \ref{sec:var}). We first treat the case when $X$ has log canonical
(and hence normal singularities) and then the general case of a variety
$X$ with semi-log singularities. Anyway, as we recall, the latter
case reduces to the former (if one works in the setting of pairs)
if one passes to the normalization.
\item[$\bullet$] \S \ref{sec:var}: We prove the uniqueness and existence of a weak Kähler-Einstein
metric in the general setting of varieties of log general type. The
existence is proved by adapting the variational approach to complex
Monge-Ampère equations introduced in \cite{BBGZ} to the present
setting. This method produces a singular Kähler-Einstein metric with
finite energy (the new feature here compared to \cite{BBGZ} is that the reference
measure does not have an $L^{1}$ density). We also use the variational
approach to establish a stability result for the solutions to the
equations induced from an (ample) perturbation of the log canonical
line bundle on a resolution.
\item[$\bullet$] \S \ref{sec:smooth}: Here we establish the smoothness of the Kähler-Einstein
metric, produced by the variational approach, on the regular locus
of the variety $X$ (or more generally, the pair $(X,D)$). The proof
uses a perturbation argument in order to reduce the problem the the
original setting of Kobayashi and Tian-Yau, combined with \textit{a priori}
estimates. But it should be stressed that in order to control the $\mathscr C^{0}$ norms
we need to invoke the variational stability result proved in the previous
section. 
\item[$\bullet$] \S \ref{sec:app}: We give some applications to automorphism groups and
show how to deduce the Yau-Tian-Donaldson conjecture for canonically
polarized varieties from our results.
\item[$\bullet$] \S \ref{sec:outlook}: The paper is concluded with a brief outlook on possible
applications to Miyaoka-Yau types inequalities, as well as the Weil-Peterson
geometry of the moduli space of stable varieties. These applications
will require a more detailed regularity analysis of the Kähler-Einstein
metrics that we leave for the future.\\
\end{itemize}

\section{Preliminaries}
\label{sec:ke}
\label{sec:prelim}
We collect here some useful tools or notions that we are going to work with in this paper. We start with a compact Kähler manifold $X$ of dimension $n$, and we consider a class $\alpha \in H^{1,1}(X,\R)$ which is big. By definition, this means that $\alpha$ lies in the interior of the pseudo-effective cone, so that there exists a Kähler current $T\in \alpha$, that is a current which dominates some smooth positive form $\omega$ on $X$. We fix $\theta$, a smooth representative of $\alpha$.

\subsection*{The ample locus}

An important invariant attached to $\alpha$ is the \textit{ample locus} of $\alpha$, denoted $\amp$, and introduced in \cite[\textsection 3.5]{DZD}. This is the largest Zariski open subset $U$ of $X$ such that for all $x\in U$, there exists a Kähler current $T_x\in \alpha$ with analytic singularities such that $T_x$ is smooth in an (analytic) neighbourhood of $x$. Its complement, called the augmented base locus, is usually denoted by $\mathbb B_+ (\alpha)$. In the case when $\alpha=c_1(L)$ is the Chern class of a line bundle, it is known (see e.g. \cite{BBP}) that: 
\[\mathbb B _+(L) = \bigcap_{\substack{L=A+E \\ A \,  \mathsmaller{a\!m\!pl\!e,} \,  E \ge 0}} \Supp(E) \]

\subsubsection*{Currents with minimal singularities}

We will be very brief about this well-known notion, and refer e.g. to \cite[\textsection 2.8]{DZD}, \cite[\textsection 1]{BBGZ}, \cite{Ber2} or \cite{BD} for more details and recent results. 

\noindent
By definition, if $T,T'$ are two positive closed currents in the same cohomology class $\alpha$, we say that $T$ is less singular than $T'$ if the local potentials $\vp, \vp'$ of these currents satisfy $\vp' \le \vp + O(1)$. It is clear that this definition does not depend on the choice of the local potentials, so that the definition is consistant. In each (pseudo-effective) cohomology class $\alpha$, one can find a positive closed current $T_{\rm min}$ which will be less singular than all the other ones; this current is not unique in general; only its class of singularities is. Such a current will be called current with minimal singularities.  

One way to find such a current is to pick $\theta \in \alpha$ a smooth representative, and define then, following Demailly, the upper envelope 
\[V_{\theta}:= \sup \{\vp \,\,  \theta \mathrm{-psh} , \,  \vp \le 0 \,  \, \textrm{on} \, \, X\}\] 
Once observed that $V_{\theta}$ is $\theta$-psh (in particular upper semi-continuous), it becomes clear that $\theta+ \ddc V_{\theta}$ has minimal singularities. \\

\subsection*{Non-pluripolar Monge-Ampère operator}

In the paper \cite{BEGZ}, the authors define the non-pluripolar product $T \mapsto \la T^n \ra $ of any closed positive $(1,1)$-current $T\in \alpha$, which is shown to be a well-defined measure on $X$ putting no mass on pluripolar sets, and extending the usual Monge-Ampère operator for Kähler forms (or having merely bounded potentials, cf \cite{BT}). Let us note that when $T$ is a smooth positive form $\om$ on a Zariski dense open subset $\Omega \subset X$, then its Monge-Ampère $\la T^n \ra $ is simply the extension by $0$ of the measure $\om^n$ defined on $\Omega$. 

Given now a $\theta$-psh function $\vp$, one defines its non-pluripolar Monge-Ampère by 
$\MA(\vp) := \la (\theta + \ddc \vp)^n \ra$. Then one can check easily from the construction that the total mass of $\MA(\vp)$ is less than or equal to the volume $\vol(\alpha)$ of the class $\alpha$ (cf \cite{Bou02}):
\[ \int_X \MA(\vp) \le \vol(\alpha)\]

A particular class of $\theta$-psh functions that appears naturally is the one for which the last inequality is an equality. We will say that such functions (or the associated currents) have \textit{full Monge-Ampère mass}. For example, $\theta$-psh functions with minimal singularities have full Monge-Ampère mass (cf \cite[Theorem 1.16]{BEGZ}).\\

\subsection*{Plurisubharmonic functions on complex spaces}

Here again, we just intend to give a short overview of the extension of the pluripotential theory to (reduced) complex Kähler spaces. A very good reference is \cite{Dem85}, or \cite[\S 5]{EGZ} which is  written in relation to singular Kähler-Einstein metric. We also refer to the preliminary parts of \cite{Var} or \cite{FS}.

The data of a reduced complex space $X$ includes the data of the sheaves of continuous and holomorphic functions. So the first object we would like to give a sense to is the sheaf $\mathscr C^{\infty}_X$ of smooth functions. It may be defined as the restriction of smooth functions in some local embeddings of $X$ in some $\CC^n$. One defines similarly the sheaves of smooth $(p,q)$-forms $\mathscr A_X^{p,q}$ which carry the differentials $d, \partial, \bar \partial$ satisfying the usual rules; the space of currents is by definition the dual of the space of differential forms as in the smooth case. The sheaves complexes that are induced (Dolbeault, de Rham, etc.) are however not exact in general. 

Another important sheaf is the one of pluriharmonic functions. They are defined to be smooth functions locally equal to the imaginary part of some holomorphic functions. One can show (see e.g \cite{FS}) that a continuous function which is pluriharmonic on $\xreg$ in the usual sense is automatically pluriharmonic on $X$. We denote by $\mathcal{PH}_X$ the sheaf of real-valued pluriharmonic functions on $X$.\\

Let us move on to psh functions now. There are actually two possible definitions which extend the usual one for complex manifolds. The first one, introduced by Grauert and Remmert, mimics the one in the smooth case: we will say that a function $\vp:X\to \R\cup\{-\infty\}$ is plurisubharmonic if it is upper semi-continuous and if for all holomorphic map $f:\Delta \to X$ from the unit disc in $\CC$, the function $\vp \circ f$ is subharmonic. 

We could also introduce a more local definition: a function $\vpd$ is strongly plurisubharmonic if in any local embeddings $i_{\alpha}: X \supset U_{\alpha} \hookrightarrow \CC^n$, $\vp$ is the restriction of a psh function defined an open set $\Omega_i\subset \CC^n$ containing $i_{\alpha}(U_{\alpha})$, if $X=\cup_{\alpha} U_{\alpha}$ is an open covering. 

Clearly, a strongly psh function is also psh. Actually, Fornaess and Narasimhan \cite{FN} showed that these notions coincide: a function on $X$ is psh if and only if it is strongly psh. 
On \textit{normal} spaces one still has a Riemann extension theorem for psh functions, thanks to \cite{GR}. More precisely, if $X$ is normal, $Y\subsetneq X$ is any proper analytic subspace, and $\vp: X\setminus Y \to \R\cup\{-\infty\}$ is psh, then $\vp$ extends to a (unique) psh function on $X$ if and only if it is locally bounded above near the points of $Y$, condition which is always realized if $Y$ has codimension at least two in $X$. In particular, if $X$ is normal, the data of a psh function on $X$ is equivalent to the data of a psh function on $\xreg$.

Moreover, one can show (cf \cite[Lemma 3.6.1]{BEG}) that a pluriharmonic function on $\xreg$ automatically extends to a pluriharmonic function on $X$. 

On non-normal spaces, one has to be more cautious, and it is convenient to introduce the notion of weakly psh function. Let $X$ be a reduced complex space, and $\nu:X^{\nu}\to X$ its normalization. We say that a function $\vpd$ is weakly psh if $\nu^*\vp=\vp\circ \nu$ is psh. It is not hard to see that a weakly psh function $\vp$ induces a bona fide psh function on $\xreg$ which is locally bounded from above near the points of $\xsing$. Conversely, any psh function on $\xreg$ which is locally upper bounded extends to a weakly psh function on $X$. On a normal space, a weakly psh function is of course psh, but in general these notions are different: consider $X=\{zw=0\}\subset \CC^2$, and $\vp(x)=0$ or $1$ according to the connected component of $x\in X$. We refer to \cite[Théorème 1.10]{Dem85} for equivalent characterizations of weakly psh functions and conditions on a weakly psh function that ensure that it is already psh.  

Finally, one can check that a (strongly) psh function $\vp$ on a complex space $X$ is always locally integrable with respect to the area measure induced by any local embedding of $X$ in $\CC^n$ (note that this is stronger than saying that $\vp$ is locally integrable on $\xreg$ with respect to some volume form). Moreover, a locally integrable function $\vp$ is (almost everywhere) weakly psh is and only if it is locally bounded from above and $\ddc \vp$ is a positive current. \\

\subsection*{Weights and Chern classes}

From now on, $X$ will be a \textit{normal} complex space unless stated otherwise. 

The definition of a (smooth) Kähler form is rather natural: it is a smooth real $(1,1)$-form written locally as $\ddc \psi$ for some (smooth) strictly psh function $\psi$; equivalently this is locally the restriction of a Kähler form in a embedding in $\CC^n$. Note that we could interpret this definition in terms of hermitian metrics on the Zariski tangent bundle of $X$, cf \cite{Var}. 

Let us now consider a line bundle $L$ on $X$. A smooth hermitian metric $h$ on $L$ is defined as in the smooth case: using trivialisations $\tau_{\alpha} : L_{|U_{\alpha}} \overset{\simeq}{\longrightarrow} U_{\alpha} \times \CC$, we just ask $h$ to be written as $h(v)=|\tau_{\alpha}(v)|^2 e^{-\vp_{\alpha}(z)}$ where $\vp_{\alpha}$ is a smooth function on $U_{\alpha}$. We say that the data $\phi:=\{(U_{\alpha}, \vp_{\alpha})\}$ is a weight on $L$, so that it is equivalent to consider a weight or an hermitian metric. 

Observe that if $(g_{\alpha \beta}:U_{\alpha}\cap U_{\beta} \to \CC^*)$ is the cocycle in $H^{1}(X, \mathcal O_X^*)$ determined by the $\tau_{\alpha}$'s (more precisely $\tau_{\alpha}\circ \tau_{\beta}^{-1}(z,v)=(z,g_{\alpha \beta}v)$), then we have necessarily $\vp_{\beta}-\vp_{\alpha}=\log |g_{\alpha \beta}|^2$. In particular, the forms $\ddc \vp_{\alpha}$ glue to a global smooth $(1,1)$-form on $X$ called curvature of $(L,h)$ and denoted by $c_1(L,h)$. 
This forms lives naturally in the space $H^0(X, \mathscr C^{\infty}_X/\mathcal{PH}_X)$, and using the exact sequence
\[0 \longrightarrow   \mathcal{PH}_X \longrightarrow \mathscr C^{\infty}_X  \longrightarrow  C^{\infty}_X/\mathcal{PH}_X  \longrightarrow 0\]
one may attach to $(L,h)$ a class $\hat c_1(L,h) \in H^1(X, \mathcal{PH}_X) $. It is then easy to see that this class actually does not depend on the choice of $h$, so we will denote it by $\hat c_1(L)$. If $X$ is smooth, $H^1(X, \mathcal{PH}_X) \simeq H^{1,1}(X,\R)$, and it is well-known that $\hat c_1(L)$ coincides with the image of $L \in H^1(X,\mathcal O_X^*)$ in $H^2(X,\Z)$ via the connecting morphism induced by the exponential exact sequence 
\[0  \longrightarrow \Z   \longrightarrow \mathcal O_X   \overset{e^{2i\pi \cdotp}}{\longrightarrow} \mathcal O_X^*   \longrightarrow  0\]
This sequence also exists on any (even non-reduced) complex space, so that $c_1(L) \in H^2(X,\Z)$ is well-defined; it will be more convenient for us to look at the image of $c_1(L)$ in $H^2(X,\R)$ however. To relate it to $\hat c_1(L)$, we may use the following exact sequence: 
\begin{equation}
\label{eq:es}
0  \longrightarrow \R  \longrightarrow \mathcal O_X   \overset{-2 \mathrm{Im}(\cdotp)}{\longrightarrow} \mathcal{PH}_X     \longrightarrow  0
\end{equation}
It is not hard to check that the connecting morphism $H^1(X, \mathcal{PH}_X) \to H^{2}(X,\R)$ sends $\hat c_1(L)$ to $c_1(L)$ as expected. \\

We will also have to consider singular weights, which are by definition couples $\phi:=\{(U_{\alpha}, \vp_{\alpha})\}$ where $U_{\alpha}$ is covering of $X$ trivializing $L$, and $\vp_{\alpha}$ are locally integrable on $U_{\alpha}$, satisfying $\vp_{\beta}-\vp_{\alpha}= \log |g_{\alpha \beta}|^2$ on $U_{\alpha}\cap U_{\beta}$. The associated curvature current, denoted by $\ddc \phi$, is well-defined on $X$. The weight is said psh if the $\vp_{\alpha}$ are, in which case $\ddc \phi$ is a positive current. Moreover, we can proceed as in the smooth case to attach to $\ddc \phi$ a class $\hat c_1(L) \in H^1(X, \mathcal{PH}_X)$ (consider $\phi$ as a section of the sheaf $L^1_{\rm loc}/\mathcal{PH}_X$ and use the natural exact sequence), whose image in $H^2(X ,\R)$ via the long exact sequence in cohomology induced by \eqref{eq:es} is $c_1(L)$. Therefore when $\phi$ is a singular weight on $L$, we may say that $\ddc \phi$ is a current in $c_1(L)$. \\

We claim that a (possibly singular) psh weight $\phi$ on $L_{|\xreg}$ \-- and thus a psh weight in the usual sense\-- automatically extends to a (unique) psh weight $\widetilde{\phi}$ on $L$. Indeed, by Grauert and Remmert's theorem, the $\vp_{\alpha}$'s defined on $U_{\alpha}\cap \xreg$ extend to a psh function $\widetilde{\vp_{\alpha}}$ on the whole $U_{\alpha}$, which is moreover defined by $\widetilde{\vp_{\alpha}}(z_0) = \limsup_{\xreg \ni z \to z_0} \vp_{\alpha}(z)$. Therefore, the relation $\widetilde{\vp_{\beta}}=\widetilde{\vp_{\alpha}}+ \log |g_{\alpha \beta}|^2$ is immediately satisfied on the whole $U_{\alpha}\cap U_{\beta}$, which proves the claim. \\

If we get back to the non-normal case, we can define psh weights using weakly psh functions instead of psh functions. So the general philosophy is that we always pull-back our objects to the normalization where things behave better, and the notions downstairs are defined and studied upstairs. For example, we can define on a normal variety the analogue of the non-pluripolar product and consider Monge-Ampère equations as in the smooth case (cf \cite[\S 1.1-1.2]{BBEGZ}. Then, if we write them on a non-normal variety, they have to be thought as pulled-back to the normalization.

\subsection*{Log canonical pairs}

Following the by now common terminology of Mori theory and the minimal model program (cf e.g. \cite{KM}), a pair $(X,D)$ is by definition a complex normal projective variety $X$ carrying a Weil $\Q$-divisor $D$ (not necessarily effective). We will say that the pair $(X,D)$ is a log canonical pair if $K_X+D$ (which is a priori defined as a Weil divisor) is $\Q$-Cartier, and if for some (or equivalently any) log resolution $\pi : X'\to X$, we have: 
\[K_{X'}=\pi^*(K_X+D)+\sum a_i E_i\]
where $E_i$ are either exceptional divisors or components of the strict transform of $D$, and the coefficients $a_i$ satisfy the inequality $a_i\ge -1$. \\

\section{Singular Kähler-Einstein metrics}
\label{sec:singke}
\subsection{Kähler-Einstein metrics on pairs}

In this section, we will consider log pairs $(X,D)$ where $X$ is a complex normal projective variety, $D$ is a Weil divisor, 
and $K_X+D$ is assumed to be $\Q$-Cartier. We choose a psh weight $\phi_D$ on $\xreg$ satisfying $\ddc \phi_D= [D_{|\xreg}]$.
The first definition concerns the Ricci curvature of currents: 

\begin{defi}
Let $\om$ be a positive current on $\xreg$; we say that $\om$ is admissible if it satisfies:
\begin{enumerate}
\item Its non-pluripolar product $\la \om^n \ra$ defines a (locally) absolutely continuous measure on $\xreg$ with respect to $d\mathbf{z}  \wedge d \bar {\mathbf {z}}$, where $\mathbf {z}= (z_i)$ are local holomorphic coordinates.
\item The function $\log(\la \om^n\ra/d\mathbf{z}  \wedge d \bar {\mathbf {z}}) $ belongs to $ L^1_{\rm loc}(\xreg) .$
\end{enumerate}
In that case, we define (on $\xreg$) the Ricci curvature of $\om$ by setting $\Ric \om := - \ddc \log \la \om^n \ra$. 
\end{defi}

Another way of thinking of this is to interpret the positive measure $\la \om^n \ra_{|\xreg}$ as a singular metric on $-K_{\xreg}$ whose curvature is $\Ric \om$ by definition.

\subsubsection*{The measure $e^{\phi}$ for $\phi$ a weight on $K_X$}
In the same  spirit, we will use the convenient but somehow abusive notation $e^{\phi}$ for $\phi$ a weight on $K_X$ (whenever the latter is defined as a $\Q$-line bundle) to refer to the positive measure $e^{\vp_{\mathbf z}} d\mathbf{z}  \wedge d \bar {\mathbf {z}}$ defined on $\xreg$ and extended by $0$ to $X$; where $\vp_{\mathbf z}$ is the expression on some trivializing chart of $\xreg$ (and hence of $K_{\xreg}$ too) of $\phi$. In particular, for $\phi$ a psh weight on $K_X+D$, the measure $e^{\phi-\phi_D}$ can easily be pulled-back to any log resolution $(\pi,X',D')$ of $(X,D)$ (we pull it back over $\xreg\setminus \Supp(D)$ and then extend it by $0$), where it become $e^{\phi\circ \pi-\phi_{D'}}$.\\

We may now introduce the notion of (negatively curved) Kähler-Einstein metric attached to a pair $(X,D)$: 

\begin{defi}
\phantomsection
\label{def:ke}
Let $(X,D)$ be a log pair; we say that a positive admissible current $\om$ is a Kähler-Einstein metric with negative curvature for $(X,D)$ if: 
\begin{enumerate}
\item $\Ric \om = -\om + [D]$ on $\xreg$,
\item $\int_{\xreg} \la \om^n \ra  = c_1(K_X+D)^n$.
\end{enumerate}
\end{defi}

This conditions are sufficient to show that a Kähler-Einstein metric is a global solution of a Monge-Ampère equation. More precisely, we have the following: 

\begin{prop}
Let $(X,D)$ be a log pair, and $\om$ be a Kähler-Einstein metric for $(X,D)$. Then $\phi:=  \log \la \om^n \ra +\phi_D$ extends to $X$ as a psh weight with full Monge-Ampère mass on $K_X+D$, solution of
\[\la(\ddc \phi)^n\ra=e^{\phi-\phi_D}\]
Conversely, any psh weight $\phi$ on $K_X+D$ with full Monge-Ampère mass solution of the equation induces a Kähler-Einstein metric  $\om:=\ddc \phi$ for $(X,D)$.
\end{prop}

\begin{proof}
On $\xreg$, we have $\ddc \phi= \om$ thus $\phi$ is a psh weight on $(K_X+D)_{|\xreg}$, and thanks to a theorem of Grauert and Remmert, it extends through $\xsing$ which has codimension at least $2$. Clearly, we have $\om= \ddc \phi$ on $X$, and by condition $3.$ in the definition of a Kähler-Einstein metric, $\phi$ has full Monge-Ampère mass. Then by definition, the two (non-pluripolar) measures $\la (\ddc \phi)^n \ra$ and $e^{\phi-\phi_D}$ coincide.

For the converse, let $\om:= \ddc \phi$; clearly $1.$ and $3.$ are satisfied. Moreover, $\phi$ and $\phi_D$ are locally integrable, so that $\om$ is admissible and $\Ric \om =-\ddc (\phi-\phi_D)=  -\om +[D]$.
\end{proof}

This proposition shows that the different definitions of what should be a singular Kähler-Einstein metric, appearing e.g. in \cite{rber, BEGZ, CGP, EGZ} etc. coincide. Moreover, one could equally define positively curved Kähler-Einstein metrics in an equivalent way as in \cite{BBEGZ}. In particular this objects, intrinsically defined on $X$, can also be seen on any log resolution in the usual way; in practice, we will most of the time work on log resolutions when dealing with existence or smoothness questions.\\

Note also that we could have chosen to define a Kähler-Einstein metric attached to a pair $(X,D)$ (say satisfying $K_X+D$ ample) to be a \textit{smooth} Kähler metric $\om$ on $\xreg \setminus \Supp(D)$ which extends to an admissible current on $\xreg$ satisfying there $\Ric \om = -\om + [D]$ and the mass condition $\left(\int_{\xreg} \la \om^n \ra =\right) \int_{\xreg \setminus \Supp(D)} \om^n =  c_1(K_X+D)^n$. 

Then, our regularity Theorem (say combined with Proposition \ref{prop:kelc}) shows \textit{a posteriori} that this definition would have coincided with Definition \ref{def:ke}.\\

Let us also mention that in the case of a log smooth log canonical pair $(X,D)$, the same proof as \cite[Proposition 2.5]{G12} combined with \cite{GW} will show that the data of a negatively curved Kähler-Einstein on $(X,D)$ is equivalent to giving an admissible current $\om$ on $X\setminus \Supp(D)$ such that: 
\begin{enumerate}
\item[$\cdotp$] $\Ric \om = - \om$ on $X\setminus \Supp(D)$, 
\item[$\cdotp$] There exists $C>0$ such that \[C^{-1} dV \le \prod_{a_j <1}  |s_j|^{2a_j}\cdotp\prod_{a_k=1}(|s_k|^2 \log^2 |s_k|^2) \, \,  \om^n \le \, C dV\]
for some volume form $dV$ on $X$, and where $D=\sum a_i D_i$, $s_i$ being a defining section of $D_i$. \\
\end{enumerate}

\subsection{Kähler-Einstein metrics on stable varieties}

Stable varieties, as considered e.g. in \cite{KSB, Karu, Kollar,Kov} are the appropriate singular varieties to look at if one wants to compactify the moduli space of canonically polarized projective varieties (cf also \cite{Vie}). Before giving the precise definition of a stable variety, we explain very briefly that notion and give the connection with Kähler-Einstein theory. In the next section, we will give a more detail account of the type of singularities involved. \\

So first of all, we will consider complex varieties that are Gorenstein in codimension $1$ (this condition replaces regularity in codimension $1$ for normal varieties) and satisfy the condition $S_2$ of Serre. Basically, the singularities in codimension $1$ of our varieties are those of the union of two coordinate hyperplanes ("double crossing"), so it is important to be aware that such varieties are in general not irreducible, and hence their normalization will not be connected. 

Now we want to recast them in the context given by the singularities of the minimal model program (MMP); so we consider such a variety $X$ and its normalization $\nu:X^{\nu}\to X$. One can write $\nu^*K_X=K_{X^{\nu}}+D$ for some reduced divisor $D$ called the conductor of $\nu$; its sits above the codimension $1$ component of the singular locus of $X$. We then say that $X$ has semi log canonical singularities if the pair $(X^{\nu}, D)$ is log canonical in the usual sense. The generalization of the notion of stable curve is given by the following definition:

\begin{defi}
A projective variety $X$ is called stable if $X$ has semi-log canonical singularities, and $K_X$ is an ample $\Q$-line bundle. 
\end{defi}

There is a subtlety for the definition of $K_X$, but we refer to \S \ref{sec:sing} for appropriate explanations. Its is actually possible to define the notion of Kähler-Einstein metric for a stable variety: 

\begin{defi}
Let $X$ be a stable variety. A Kähler-Einstein metric on $X$ is a positive admissible current $\om$ on $\xreg$ such that: 
\begin{enumerate}
\item $\Ric \om = -\om$ on $\xreg$,
\item $\int_{\xreg} \la \om^n \ra  = c_1(K_X)^n$.
\end{enumerate}
\end{defi}

In the non-normal case however, psh weight do not automatically extend across the singularities, so that it is not clear that the Kähler-Einstein metric will extend as a positive current on $K_X$ satisfying the usual Monge-Ampère equation \textit{globally}. Actually, 
this is the case as shows the following proposition: 
\begin{prop}
\phantomsection
\label{prop:ext}
Let $X$ be a stable variety, and $\om$ a Kähler-Einstein metric on $X$. Then the weight $\phi:=\log \om^n$ extends to $X$ as a weakly psh weight on $K_X$ solution of the Monge-Ampère equation $\la(\ddc \phi)^n\ra = e^{\phi}$.
\end{prop}

\begin{proof}
Taking the $\ddc$ of each side in the definition of $\phi$ and using the Ricci equation, we find $\om=\ddc \phi$ on $\xreg$, and therefore $\phi$ satisfies $\la (\ddc \phi) ^n \ra = \la  \om ^n \ra =e^{\phi}$. Pulling back this equation to normalization $\Xnu$, we find a psh weight $\phi'=\nu^* \phi $ on $c_1(\nu^* K_X)_{|\nu^{-1}(\xreg)}$ solution of $\la (\ddc \phi')^n \ra = e^{\phi'-\phi_D}$ where $D$ is the conductor of the normalization. As we work inside $\Xnu_{reg}$ and the integral $\int_{\nu^{-1}(\xreg)} e^{\phi'-\phi_D}$ is finite, we infer from Lemma \ref{lem:ext} below that $\phi'$ extends (as a psh weight) across $D_{\rm reg}$. So $\phi'$ induces a psh weight on $c_1(\nu^* K_X)_{|\Xnu_{\rm reg}\setminus D_{\rm sing}}$, and by normality of $\Xnu$, it extends to the whole $\Xnu$, which means precisely that $\phi$ extends as a weakly psh weight on $K_X$. The expected Monge-Ampère equation holds automatically on $X$ (or equivalently on $X'$) since both measures $\la(\ddc \phi)^n\ra$ and $e^{\phi}$ put no mass on $\xsing$ by definition.
\end{proof}

In the previous proof, we used the following extension result:  

\begin{lemm}
\phantomsection
\label{lem:ext}
Let $U$ be a neighbourhood of $0\in \CC^n$, $H=\{z_1=0\} \subset \CC^n$, and $\vp$ be a psh function defined on $U \setminus H$. We assume that the integral
\[\int_{U\setminus H} \frac{e^{\vp}}{|z_1|^2}dV\]
is finite. Then $\vp$ extends across $H$, and more precisely $\vp$ tends to $-\infty$ near $H$.
\end{lemm}

\begin{proof}
(thanks to Bo Berndtsson for providing us with this elegant proof) Assume, to get a contradiction, that $\vp$ does not tend to $-\infty$ near $H$, and let $V:=U\setminus H$. Then we can find a sequence $(x_k)$ of points in $V$ converging to $H$ such that $\vp(x_k) \ge -C$ for some constant $C$. We write $x_k=(x_{1,k}, \ldots, x_{n,k})$, and we set $r_k=|x_{1,k}|/2$; the sequence $(r_k)$ converges to $0$, and if $D_k$ denotes the polydisk centered at $x_k$ with polyradius $(r_k, \delta, \ldots, \delta)$ for some fixed $\delta >0$, then we have $D_k \subset V$. 

\noindent
Using the mean value inequality for $\vp$ at $x_k$, we find: 
\[-C \le \frac{1}{\mathrm{vol}(D_k)}\int_{D_k} \vp \,  dV\]
Therefore, using Jensen's inequality, we obtain, up to modifying $C$ by a normalization factor depending only on the dimension $n$:
\[e^{-C} \le \int_{D_k} \frac{e^{\vp}dV}{r_k^2 \delta^{2(n-1)}}\]
but on $D_k$, $|z_1| \le 3 r_k$ so 
\[e^{-C'} \le \int_{D_k}\frac{e^{\vp}dV}{|z_1|^2} \]
for $C'= C+\log 9-2(n-1) \log \delta$. As the measure of $D_k$ goes to zero when $k\to +\infty$, it shows that the integral $\int_{U\setminus H} \frac{e^{\vp}}{|z_1|^2}dV$ is infinite, which is absurd.
\end{proof}

\subsection{Singularities of stable varieties}
\label{sec:sing}
In this paragraph, we intend to give a more precise overview of the notion of semi-log canonical singularities. As we will just touch on this topic, we refer to the nice survey \cite{Kov} for a broader study. Other good references are \cite{Kollar, KSB}.\\
In the following, $X$ will always be a reduced and equidimensional scheme of finite type over $\CC$, and we set $n:=\dim X$. We emphasize again on the fact that $X$ will not be irreducible in general.

\subsubsection*{The conditions $G_1$ and $S_2$}
As we saw earlier, we need a canonical sheaf. The condition $G_1$ will guarantee its existence, and the condition $S_2$ will (among other things) ensure its uniqueness. 

If $X$ is Cohen-Macaulay (for every $x\in X$, the depth of $\mathcal O_{X,x}$ is equal to its Krull dimension), then $X$ admits a dualizing sheaf $\om_X$. We say that $X$ is Gorenstein if $X$ is Cohen-Macaulay and $\om_X$ is a line bundle. We say that $X$ is $G_1$ if $X$ is Gorenstein in codimension $1$, which means that there is an open subset $U \subset X$ which is Gorenstein and satifies $\textrm{codim}_X(X\setminus U) \ge 2$. 

We say that $X$ satisfies the condition $S_2$ of Serre if for all $x\in X$, we have $\textrm{depth}(\mathcal O_{X,x}) \ge \min\{\textrm{ht} ( \mathfrak m_{X,x}), 2\}$, where $\textrm{ht}( \mathfrak m_{X,x})=\mathrm{codim}(\bar x)$ denotes the height of the maximal ideal $\mathfrak m_{X,x}$ of $\mathcal O_{X,x}$. This condition is equivalent to saying that for each closed subset $i: Z\hookrightarrow X$ of codimension at least two, the natural map $\mathcal O_{X}\to i_*\mathcal O_{X\setminus Z}$ is an isomorphism. 

If $X$ is $G_1$ and $S_2$, and $U\subset X$ is a Gorenstein open subset whose complement has codimension at least $2$, one can then define the canonical sheaf $\om_X$ by $\om_X:= j_* \om_U$ where $j: U \hookrightarrow X$ is the open embedding, and $\om_U$ is the dualizing sheaf of $U$. 
By definition, this is a rank one reflexive sheaf. When $X$ is projective, we know that it admits a dualizing sheaf; as it is reflexive, it coincides with $\om_X$ by the $S_2$ condition.

We would like to have an interpretation of $\om_X$, or at least $\om_U$ in terms of Weil divisor as in the normal case where we define the Weil divisor $K_X$ as the closure of 
some Weil divisor representing the line bundle $K_{\xreg}$. But we have to be more cautious in the non normal case it is not clear how we should extend a Weil divisor given on $\xreg$. Actually, this is where the $G_1$ conditions appears: as $\om_U$ is a line bundle, or equivalently a Cartier divisor, we may choose a Weil divisor $K_U$ whose support does not contain any component of $\xsing$ of codimension $1$ and represent $\om_U$ (write $\om_U$ as the difference of two very ample line bundles). Then we define $K_X$ to be the closure of $K_U$. Clearly, the divisorial sheaf $\mathcal O_X(K_X)$ is reflexive, and coincides with $\om_U={\om_X}_{|U}$ on $U$, so that by the $S_2$ condition, we get: 
\[\om_X \simeq \mathcal O_X(K_X)\]
In fact, if $\om_X^{[m]}$ denotes the $m$-th reflexive power of $\om_X$, the same arguments yield $\om_X^{[m]}\simeq \mathcal O_X(mK_X)$. Therefore, the Weil divisor is $\Q$-Cartier if and only if $\om_X$ is a $\Q$-line bundle, i.e. $\om_X^{[m]}$ is a line bundle for some $m>0$. 

\subsubsection*{Conductors and slc singularities}
\label{sec:cond}
Let now $X$ be a (reduced) scheme, and $\nu: \Xnu \to X$ its normalization. We recall that if $X$ is not irreducible, its normalization is defined to be the disjoint union of the normalization of its irreducible components. The \textit{conductor ideal}
\[\textrm{cond}_X:= \mathscr{H}\!om_{\mathcal O_X}(\nu_* \mathcal O_{\Xnu}, \mathcal O_X)\]
is the largest ideal sheaf on $X$ that is also an ideal sheaf on $\Xnu$. If we think of the case where $B$ is the integral closure of some integral ring $A$, then we can easily see that $\textrm{Hom}_A(B,A)$ injects in A (via the evaluation at $1$), and the image of this map is the annihilator $\mathrm{Ann}_A(B/A)=\{f\in A; fB \subset A\}$, or equivalently the largest ideal $\mathcal I\subset A$ that is also an ideal in $B$. 

Coming back to the case of varieties, we will denote by $\textrm{cond}_{\Xnu}$ the conductor seen as an ideal sheaf on $\Xnu$, and we define the conductor subschemes as $C_X:= \mathrm{Spec}_X(\Ox/\cond_X)$ and 
$C_{\Xnu}:= \mathrm{Spec}_{\Xnu}(\mathcal O_{\Xnu}/\cond_{\Xnu})$. If $X$ is $S_2$, then one can show that these schemes have pure codimension $1$ (and hence define Weil divisors) but they are in general not reduced (e.g. the cusp $y^2=x^3$). 

If $K_X$ is $\Q$-Cartier and $X$ is demi-normal (i.e. $X$ is $S_2$ and has only double crossing singularities in codimension $1$, cf \cite{Kollar}), we have the following relation: 
\begin{equation}
\label{eq:can}
\nu^*K_X= K_{\Xnu}+ C_{\Xnu}
\end{equation}
The proof of this identity goes as follows: first, using the demi-normality assumption, we may assume that the only singularities of $X$ are double normal crossings. Then, using the universal property of the dualizing sheaf (which coincide with the canonical sheaf as we observed above) and 
the projection formula, we have $\nu_* \om_{\Xnu}=\om_X(-C_X)$. We pull-back this relation to $\Xnu$ using the fact that the sheaf $\Ox(-C_{X})$ becomes precisely $\mathcal O_{\Xnu}(-C_{\Xnu})$. By the assumptions on the singularities, this last sheaf is actually an invertible sheaf so that we get the expected identity (cf point $8$ in \cite{Kollar}). As we will explain below, we do not want to assume a priori that our varieties are demi-normal. Therefore, it may happen that $C_{\Xnu}$ is not Cartier, and the formula \eqref{eq:can} may not be true anymore. So whenever we will deal with Kähler-Einstein on those varieties, we will have to apply the arguments on a log-resolution of the normalization instead of the normalization itself. Anyway, this will not cause any troubles.  

An important point is that whenever the conductor is reduced, then necessarily $X$ is seminormal (i.e. every finite morphism $X'\to X$ (with $X'$ reduced) that is a bijection on points is an isomorphism); moreover, a seminormal scheme which is $G_1$ and $S_2$ has only double crossing singularities in codimension 1, i.e. it is demi-normal. We refer to \cite{Tra,GT,KSS}) for the previous assertions. This leads to the following definition: 

\begin{defi}
\phantomsection
\label{def:slc}
We will say that $X$ has semi-log canonical singularities if:
\begin{enumerate}
\item $X$ is $G_1$ and $S_2$,
\item $K_X$ is $\Q$-Cartier,
\item The pair $(\Xnu, C_{\Xnu})$ is log-canonical.
\end{enumerate}
\end{defi}

If $X$ has semi-log canonical singularities (slc), then $C_{\Xnu}$ is necessarily reduced, and therefore the codimension $1$ singularities of $X$ are only double crossing as we explained above. This assumption
is usually added in the definitions (cf \cite{Kollar, Kov}), but we may keep it or not without any change. This justifies the seemingly different definition given in the previous section. Finally, we can give the definition of a stable variety:

\begin{defi}
We say that $X$ is stable if
\begin{enumerate}
\item $X$ is projective,
\item $X$ has semi-log canonical singularities, 
\item $K_X$ is $\Q$-ample.
\end{enumerate}
\end{defi}

\subsubsection*{Singularities and Kähler-Einstein metrics: a summary}

If we take a closer look at the proof of Proposition \ref{prop:ext}, we see that we did not use all of the properties of a stable variety to see that a Kähler-Einstein metric always extend. Actually, we just used the fact that the conductor was a divisor. Therefore, using the existence and regularity results that we are going to prove in the next sections, and the restriction on the singularities of a pair carrying a Kähler-Einstein metric (cf Proposition \ref{prop:kelc}), we can summarize the problem of the existence of a Kähler-Einstein metric on a stable variety in the following statement: 

\begin{theo}
\phantomsection
\label{thm:sumup}
Let $X$ be a reduced $n$-equidimensional projective scheme over $\CC$, satisfying the conditions $G_1$ and $S_2$, and such that $K_X$ is an ample $\Q$-line bundle. Then the following are equivalent: 
\begin{enumerate}
\item There exists a Kähler form $\om$ on $\xreg$ such that $\Ric \om = -\om$ and $\int_{\xreg} \om^n = c_1(K_X)^n$,
\item There exists $\om$ as above which extends to define a positive current in $c_1(K_X)$,
\item $X$ has semi-log canonical singularities, i.e. $X$ is stable.
\end{enumerate}

Moreover, by the results of Odaka \cite{Odaka, od2}, the latter condition is equivalent to: 

\textit{4.  The pair $(X, K_X)$ is $K$-stable}.
\end{theo}

\section{Variational solutions}
\label{sec:var}

\subsection{General setting}
\label{ssec:gen}
Consider the following general setting: $X$ is a compact Kähler manifold
and $[\omega]$ a big class, with $\omega$ smooth (but not necessarily positive). We say that
a function $u\in PSH(X, \om)$ has \emph{ full Monge-Ampère mass,} and we will write $u\in\mathcal{E}(X,\omega),$
if the total mass of $\MA(u)$ is equal to the volume of the class
$[\omega],$ where the volume in question may be defined by $V:=\mathrm{vol}([\omega]):=\int_{X}\MA(u_{\rm min}),$
for $u_{\rm min}$ any element in $PSH(X,\omega)$ with minimal singularities, cf \S \ref{sec:prelim}.\\

We now recall an important subspace of $\mathcal E(X, \om)$ denoted by
$\mathcal E^1(X, \om)$, and consisting of functions with finite energy. The energy $\mathcal E(u)$ of an $\om$-psh function $u$ (not necessarily in $\mathcal E(X, \om)$) is defined in the following way (cf \cite{GZ07, BEGZ, BBGZ} for more details \--- the energy is sometimes denoted by $E$ in the aforementioned papers). 

\noindent
First, if $u \in PSH(X, \om)$ has minimal singularities, we set \[\mathcal E(u):= \frac{1}{(n+1)V} \sum_{j=0}^n \int_X (u-V_{\theta}) \,  \MA(u^{(j)}, V_{\theta}^{(n-j)})\]
where $\MA$ is the mixed non-pluripolar Monge-Ampère operator. If now $u$ is any $\om$-psh function, we defined
\[\mathcal E(u):= \inf \left\{\mathcal E(v) \, | \, v \in PSH(X, \om) \, \textrm{with minimal singularities}, \, v \ge u \right\}\]
Then we set $\mathcal E^1(X, \om):=\{u \in PSH(X, \om), \, \mathcal E(u)> -\infty\}.$ Actually, \cite[Proposition 2.11]{BEGZ} gives another characterization of this last space: a function $u \in PSH(X, \om)$ belongs to $\mathcal E^1(X, \om)$ if and only if $u \in \mathcal E(X, \om)$ and $\int_X(u-V_{\theta}) \MA(u)<+\infty$ (and for any $u \in \mathcal E(X, \om)$, the explicit integral formula for  $\mathcal E(u)$ above is still valid). Using this result, it becomes clear that $\mathcal E^1(X, \om) \subset \mathcal E(X, \om)$ as announced. \\

We should finally add that $\mathcal E$ is an upper-semicontinuous (usc) concave functional on $PSH(X, \om)$, and that it is the normalized primitive of the Monge-Ampère
operator, i.e. 
\begin{equation}
\label{eq:energy}
(d\mathcal{E})_u = \frac 1 V \, \MA(u)^{n}.
\end{equation}

\subsection{Uniqueness}

Given a measure $\mu$ on $X$ (possible non-finite) we consider the following MA-equation
for $u\in PSH(X,\omega)$ attached to the pair $(\omega,\mu):$ 
\begin{equation}
\omega_{u}^{n}=e^{u}\mu,\label{eq:ma eq for weighted measure}
\end{equation}
where $\omega_{u}^{n}:=\MA(u)$ is the non-pluripolar Monge-Ampère
operator as before. This equation is equivalent
to the following \emph{normalized} MA-equation on $\mathcal{E}(X,\omega)/\R:$
\begin{equation}
\frac{\omega_{u}^{n}}{V}=\frac{e^{u}\mu}{\int e^{u}\mu},\label{eq:normalized ma eq for weighted measure}
\end{equation}
The equivalence follows immediately from the $\R-$invariance of the
latter equation and the substitution $u\mapsto u-\log\int e^{u}\mu$
which maps solutions of equation \eqref{eq:ma eq for weighted measure}
to solutions of the equation \eqref{eq:normalized ma eq for weighted measure}.
\begin{prop}
Any two solutions $u$ and $v$ of the MA-equation \eqref{eq:ma eq for weighted measure}
such that $u$ and $v$ are in $\mathcal{E}(X)$ coincide. \end{prop}
\begin{proof}
This is an immediate consequence of the comparison principle \cite[Corollary 2.3]{BEGZ}:
if $u$ and $v$ are in $\mathcal{E}(X)$ then 
\[
\int_{\{u<v\}}\MA(v)\leq\int_{\{u<v\}}\MA(u)
\]
But the MA above then forces $u=v$ a.e wrt the measure $\mu.$ Since
$\mu$ cannot charge pluripolar sets (as $\MA(u)$ does not) it follows
that $u=v$ away from a pluripolar set and hence everywhere, by basic
properties of psh functions.
\end{proof}

\subsection{Existence results for log canonical pairs}

Let $(X,D)$ be a log canonical pair such that the log canonical divisor
$K_{X}+D$ is big. Assume that $(X,D)$ is a log smooth,
i.e. $X$ is smooth and 
\[
D=\sum_{i}c_{i}D_{i}
\]
 is a normal crossings divisor with $c_{i}\in]-\infty,1].$ To the
pair $(X,D)$ we can associate the following Kähler-Einstein type
equation for a metric $\phi$ on $L:=K_{X}+D:$
\begin{equation}
(dd^{c}\phi)^{n}=e^{\phi-\phi_{D}},\label{eq:k-e eq for weights}
\end{equation}
 where $\phi_{D}=\sum_{i}c_{i}\log|s_{i}|^{2}$ and $s_{i}$ are sections
cutting out the divisors $D_{i}$ above.
\begin{theo}
\phantomsection
\label{thm:exist}
There is a unique finite energy solution $\phi$ to the equation above.
\end{theo}

\begin{proof}
The proof is a modification of the variational approach in \cite{BBGZ}
(concerning the case when $D$ is trivial). To explain this we fix
a smooth form $\omega\in c_{1}(K_{X}+D).$ Then the equation above
is equivalent to a Monge-Ampère equation for an $\omega-$psh function
$u:$ 
\begin{equation}
\omega_{u}^{n}=e^{u}\mu\label{eq:k-e equ for omega psh}
\end{equation}
 where the measure $\mu$ is of the form $\mu=\rho dV$ for a function
$\rho$ in $L^{1-\delta}(X)$ (but $\rho$ is not in $L^{1}(X)!).$
We let 
\[\mathcal{L}(u):=-\log\int e^{u}\mu\]
Then, at least formally, solutions of equation \eqref{eq:k-e equ for omega psh}
are critical points of the functional 
\[\mathcal{G}(u):=\mathcal{E}(u)\text{+}\mathcal{L}(u).\]
in view of the equation \eqref{eq:energy} satisfied by $\mathcal E$.
$\mathcal{L}$ also defines an usc concave functional on $PSH(X,\omega)$
and we let $\mathcal{L}(X,\omega):=\{\mathcal{L}>-\infty\}$ (the upper semi-continuity follows
from Fatou's lemma).

 Note that Lemma \ref{lem:ex} below guarantees that the intersection $\mathcal{E}^{1}(X,\omega)\cap\mathcal{L}(X,\omega)$
is non-empty. Hence, $\mathcal{G}(u)$ is not identically equal
to $-\infty$ on its domain of definition that we will take to be
$\mathcal{E}^{1}(X,\omega)$ (equipped with the usual $L^{1}(X)-$topology).

Next, we observe that 
\begin{equation}
\mathcal{G}(u)\leq\mathcal{E}(u)-\int u\mu_{0}+C''\label{eq:upper bound on g}
\end{equation}

Indeed, since $\mu\geq C\mu_{0},$ where $\mu_{0}$ is finite measure
on $X$ integrating all quasi-psh functions on $X$ (in our case we
may take $\mu_{0}=\left\Vert s'\right\Vert dV$ for some holomorphic
section $s'$ defined by the negative coefficients of $D):$ 
\[
\int e^{u}\mu\geq C\int e^{u}\mu_{0}
\]
and hence 
\[
\mathcal{L}(u)\leq C'-\log\int e^{u}\mu_{0}\leq C''-\int u\mu_{0}
\]
using Jensen's inequality, which proves \eqref{eq:upper bound on g}.
In particular, $\mathcal{G}(u)$ is bounded from above. Indeed, by
scaling invariance we may assume that $\sup_{X}u=0$ and then use
that, by basic compactness properties of $\omega-$psh functions,
$\sup u\leq\int u\mu_{0}+C.$ 

Let now $u_{j}\in\mathcal{E}^{1}(X,\omega)$ be a sequence such that
\[
\mathcal{G}(u_{j})\rightarrow\sup_{\mathcal{E}^{1}(X,\omega)}\mathcal{G}:=S<\infty
\]
Again, by the scale invariance of $\mathcal{G}$ we may assume that
$\sup_{X}u_{j}=0.$ In particular, 
\[
\mathcal{L}(u_{j})\geq S/2-\mathcal{E}(u_{j})
\]
for $j>j_{0}.$ But, by \eqref{eq:upper bound on g}, $\mathcal{E}(u_{j})$
is bounded from below and hence there is a constant $C$ such that
\[
\mathcal{E}(u_{j})\geq-C,\,\,\,\mathcal{L}(u_{j})\geq-C
\]
Let now $u_{*}$ be a limit point of $u_{j}.$ By upper-semicontinuity
we have that 
\[
\mathcal{E}(u)\geq-C,\,\,\,\mathcal{L}(u)\geq-C
\]
Finally, we note that $u_{*}$ satisfies the equation \eqref{eq:k-e equ for omega psh}
by applying the projection argument from \cite{BBGZ} as follows.
Fixing $v\in\mathcal{C}^{\infty}(X)$ let $f(t):=\mathcal{E}_{\omega}(P_{\omega}(u_{*}+tv))+\mathcal{L}(u_{*}+tv),$
where 

\[
P_{\omega}(u)(x):=\sup\{v(x):\, v\leq u,\,\, v\in PSH(X,\omega)\}
\]
(note that $f(t)$ is finite for any $t).$ The functional $\mathcal{L}(u)$
is decreasing in $u$ and hence the sup of $f(t)$ on $\R$ is attained
for $t=0.$ Now $\mathcal{E}_{\omega}\circ P_{\omega}$ is differentiable
with differential $MA(P_{\omega}u)$ at $u$ \cite{BBGZ}. Hence,
the condition $df/dt=0$ for $t=0$ gives that the variational equation
\eqref{eq:k-e equ for omega psh} holds when integrated against any
$v\in\mathcal{C}^{\infty}(X).$ 
\end{proof}

Let us now prove the following result, that we used in the proof: 

\begin{lemm}
\phantomsection
\label{lem:ex}

Let $(X,D)$ be a log smooth pair and $L$ a big line bundle. Let $\theta$ be a smooth $(1,1)$ form whose cohomology class is $c_1(L)$. Let $s_0$ be a section of $D$, and $|\cdot|$ a smooth hermitian metric on $\mathcal O_X(D)$. Then there exists a $\theta$-psh function $u\in \mathcal{E}^1(X,\theta)$ such that $e^{u}/|s_0|^2$ is integrable.
\end{lemm}

\begin{proof}
As $L$ is big, the exists $m$ big enough such that $mL-D$ is effective. We choose $t$ a holomorphic section of $mL-D$, and consider $s:=s_0t$ which is a section of $mL$ vanishing along $D$.
Let $h_0$ be an smooth hermitian metric on $L$ with curvature form $\theta$, and let $V_{\theta}$ be the upper envelope of all (normalized) $\theta$-psh functions.
We define on $mL$ the hermitian metric $h:=h_0^{\otimes m}e^{-mV_{\theta}}$. For $0<\alpha < 1$ small enough we claim that the function 
\[u:=V_{\theta}-\left(- \frac{1}{m}\log |s|^2_{h}\right) ^{\alpha}\]
suits our requirements.\\

First of all, it is $\theta$-psh because of the following general fact: if $\psi$ is $\theta$-psh and $\chi:\R\to \R$ is convex and non-decreasing satisfying $\chi' \le 1$, then 
$\Vt+\chi(\psi - \Vt)$ is $\theta$-psh. Indeed, $\ddc(\Vt+\chi(\psi - \Vt)) = (1-\chi')(\theta+\ddc \Vt)+ \chi' (\theta  + \ddc \psi) + \chi ''  |d(\psi-\Vt)|^2$ where $\chi'$ and $\chi''$ are evaluated at $\psi-\Vt$.
Now we apply this to $\psi = 1/m \log |s|^2_{h_0^m}$.\\

For the integrability property, we use the following inequality for $x$ a real number (big enough): $x^{\alpha}\ge (n+1) \log x -C$ for some $C>0$ depending only on $\alpha$. 
Now we observe that $\Vt+\chi(\psi - \Vt)\le \chi(\psi)$ : indeed, $\chi(\psi - \Vt)- \chi(\psi) \le \sup \chi' \cdotp (-\Vt) \le - \Vt$, so that in our case, $u \le (- \frac{1}{m}\log |s|^2_{h_0}) ^{\alpha}$. If we apply the basic inequality stated above to $x=- \frac{1}{m}\log |s|^2_{h_0}$ which can be made big enough by multiplying $h_0$ by a big constant (this does not change the curvature), we get
\[e^{u}\le C \left(- \frac{1}{m}\log |s|^2_{h_0}\right)^{-(n+1)} \]
As $D$ has snc support, and $|t|$ is bounded from above, we are left to check that the integral \[\int_D\frac{dV}{\prod_{i\le n} |z_i|^2\cdotp \log^{n+1} (\prod_{i\le n} |z_i|^2) }\] over the unit polydisc $D$ in $\CC^n$ converges. 
But after a polar change of coordinate, we are led to estimate $\int_{[0,1]^n} \frac{dx_1 \cdots dx_n}{\prod_{i\le n} x_i \cdotp \log^{n+1}( \prod_{i\le n} x_i^2)}$, which equals
$\frac{1}{2^{n+1}n}\int_{[0,1]^{n-1}}\frac{dx_1 \cdots dx_{n-1}}{\prod_{i\le n-1} x_i \cdotp \log^{n} (\prod_{i\le n-1} x_i)}$. By induction, and using the Poincaré case, it concludes.\\

Finally, one has to check that $u \in \mathcal{E}^1(X,\theta)$. We compute the capacity $\capt(u < \Vt-t)$ for $t$ big. 
But $(u < \Vt-t)= \left( \frac{1}{m}\log |s|^2_{h}<-t^{1/\al}\right) \subset \left(( \frac{1}{m}\log |s|^2_{h_0}<-t^{1/\al})\right)$ and thus 
$\capt(u < \Vt-t) \le \frac{C}{t^{1/\al}}$ because for every $\theta$-psh function $\psi$, one has $\capt(\psi<-t) \le \frac{C_{\psi}}{t}$ (this is an easy generalization of \cite[Proposition 2.6]{GZ05}).
Therefore, if $\alpha< \frac{1}{n+1}$, one has
\[\int_{0}^{+\infty}t^n \capt(u < \Vt-t)dt < +\infty\]
which, using the characterization given in \cite[Lemma 2.9]{BBGZ}, ends the proof of the lemma. 
\end{proof}

\begin{rema}
\label{rem:ep}
The proof of the preceding lemma yields actually a stronger result. If $\sum a_i \mathrm{div}(s_i)$ is an effective divisor with snc support meeting $D$ transversally and such that $a_i<1$ for all $i$, then the function $u$ obtained above satisfies $e^u/ \prod |s_i|^{2a_i} |s_0|^2 \in L^1(dV)$, and more generally this is still true for $e^{\ep u}$ for all $\ep>0$ (use the inequality $\ep x^{\alpha}\ge (n+1) \log x -C$ for $x=- \frac{1}{m}\log |s|^2_{h_0}$ this time).
\end{rema}

\subsection{Stability under perturbations}

Let now $L$ be a semipositive and big line bundle, 
and consider the perturbed ample lind bundles $L_{j}:=L+\ep_j A$, 
for $\ep_j$ a sequence of positive numbers tending to $0$ and $A$ a fixed ample line bundle. 
Fixing also a Kähler form $\om_A \in c_1(A)$ and a smooth semipositive form $\om \in c_1(L)$, we write $\om_j:=\om+\ep_j \om_A$. Let $\mu_{j}$ be
the sequence of measures on $X$ given by
\[\mu_j = \prod_{\alpha} (|s_{\alpha}|^2+\ep_j )^{e_{\alpha}} \, \frac{dV}{\prod_{\beta}|s_{\beta}|^2}\]
where $e_{\alpha}>-1$ for all $\alpha$, and the divisor $\sum_{\alpha} \mathrm{div}(s_{\alpha})+\sum_{\beta} \mathrm{div}(s_{\beta})$ is a reduced normal crossing divisor. This is precisely the sequence of approximations we are going to use to solve our Kähler-Einstein equation.

\noindent
Consider now the following Monge-Ampère equations
for $u_{j}\in\mathcal{E}(X,\omega_{j})$ (and sup-normalized): 
\[
\omega_{u_{j}}^{n}/V=\frac{e^{u_{j}}\mu_{j}}{\int_{X}e^{u_{j}}\mu_{j}}
\]
and similarly 
\[
\omega_{u}^{n}/V=\frac{e^{u}\mu}{\int_{X}e^{u}\mu}
\]
for $u\in\mathcal{E}(X,\omega).$ 
\begin{theo}
\phantomsection
\label{thm:pert}
The unique sup-normalized solution $u_{j}$ of the first equation
above converges, in the $L^{1}(X)-$topology, to the unique sup-normalized
solution $u$ to the the latter equation. Equivalently, the solutions
$v_{j}$ of the corresponding non-normalized equations converge in
$L^{1}(X)$ to $v$ solving the corresponding limiting non-normalized
equation.
\end{theo}

\begin{proof}

We denote by $\G_j$ (resp. $\L_j$) the functional determined by the pair $(\om_j, \mu_j)$ (resp. $\mu_j$), and by $u_j$ the sup-normalized maximizer of $\G_j$. We also denote by $u_0$ the sup-normalized fixed $\om$-psh function given by Lemma \ref{lem:ex}. Let us add that in the course of the proof, 
the precise value of the constant $C$ may, as usual, change from line to line. We split the proof into four steps. \\

\noindent 
\textbf{Step 1. } We first show that 
\begin{equation}
\label{eq:g}
-C \le \G_j(u_j) \le C
\end{equation}
As $u_0$ is $\om$-psh, it is also $\om_j$-psh. Moreover, the capacity computation of Lemma \ref{lem:ex} shows that the energy of $u_0$  \textit{with respect to} $\om_j$ is finite, and as $\E_{\om_j}(u_0)$ increases with $j$, we obtain
\[\E_{\om_j}(u_0) \ge -C\]
Besides, by dominated convergence, we have
$\lim_{j\to +\infty} \L_j(u_0) =  \L(u_0) $
and therefore we get $L_j(u_0) \ge -C$. Consequently, $\G_j(u_j) \ge \G_j(u_0) \ge -C$
which gives a first bound (recall that $u_j$ maximize $\G_j$ by Theorem \ref{thm:exist} and the translation invariance of $\G_j$). 

Choose now a probability measure $\mu_0$ satisying $\mu_j \ge e^{-C}\mu_0$ for all $j$ (its existence is clear given the precise form of $\mu_j$). Then Jensen's inequality gives
\[\L_j(u_j) \le -\int_X u_j d\mu_0 -C\]
but the compactness properties of quasi-psh functions (all $u_j$'s are $C\om_A$-psh) also gives the inequality
\[\sup u_j = 0 \le \int_X u_j d\mu_0 +C \]
Combining the two previous inequalities, we get
\[\G_j(u_j) \le \E_{\om_j}(u_j)+C\]
which gives both the uniform upper bound for $\G_j(u_j)$ (as $\E_{\om_j}$ is always non-positive) and a uniform lower bound $\E_{\om_j}(u_j) \ge -C$.\\

\noindent
Let $u$ be an $L^1$-limit point in $PSH(X, \om)$ of the sequence $(u_j)$. \\

\noindent 
\textbf{Step 2. } We next show that 
\begin{equation}
\label{eq:g2}
\G(u) \ge \limsup \G_j(u_j) 
\end{equation}
\noindent
First by Fatou's lemma, we have
\[\L(u) \ge  \limsup \L_j(u_j)\]
Moreoverm
\[\E(u) \ge  \limsup \E_{\om_j}(u_j)\]
as follows from Lemma \ref{lem:energie} below. Putting these two inequalities together gives the desired bound.\\

\noindent 
\textbf{Step 3. } $u$ is a sup-normalized maximizer of $\G$.\\

\noindent
For any given sup-normalized $\om$-psh function $v$, we need to show that 
\begin{equation}
\label{eq:g5}
\G(u)\ge \G(v)
\end{equation}
Of course, on can assume that $\G(v)$ is finite. Thanks to step 2, it is enough to show that $\limsup \G_j(v) \ge G(v)$. But this inequality is far from clear as we cannot directly apply the dominated convergence theorem here. Indeed, for the energy part, it could happen that $\E_{\om}(v)$ is finite though $\Ej(v)=-\infty$ for all $j$. As for the other part, despite $e^v \in L^1(\mu)$, it is not obvious that $e^v \in L^1(\mu_j)$ (because of the "zeroes" of $\mu$ which do not appear in $\mu_j$). \\

To bypass these difficulties we will use a regularization/perturbation argument. More precisely, we pick a family of smooth $\om$-psh functions $(\vd)_{\delta>0}$ which decreases to $v$, and we set for all positive $\delta, \ep$:
\[\vde:=(1-\ep) \vd+\ep u_0\]
where we recall that $u_0$ denotes the particular (sup-normalized) $\om$-psh function constructed in Lemma \ref{lem:ex}.

As $\vd$ is smooth and $u_0 \in \E^1(X, \om_j)$ for all $j$, $\vde$ has finite $\om_j$-energy, the dominated convergence theorem shows that 
\begin{equation}
\label{eq:g3}
\lim_{j\to +\infty} \E_{\omega_j} (\vde) = \E_{\om}(\vde)
\end{equation}
Moreover, as we observed in remark \ref{rem:ep}, the function $e^{\ep u_0}$ is in $L^1(\mu)$ for all $\ep>0$, and $e^{\vde} \le e^{\ep u_0}$. Therefore, by dominated convergence, we get
\begin{equation}
\label{eq:g4}
\lim_{j\to +\infty} \L_j (\vde) = \L(\vde)
\end{equation}
Combining \eqref{eq:g2} with \eqref{eq:g3} and \eqref{eq:g4}, we get $\G(u) \geq \G(\vde)$ for all $\delta, \ep >0$. Set $v_{\ep}:=(1-\ep) v+\ep u_0$. By monotonicity of $\E_{\om}$, we have $\E_{\om}(\vde) \ge \E_{\om}(v_{\ep})$. Using the dominated convergence theorem, we also see that $\L(\vde) \to \L(v_{\ep})$. Therefore, we have $\G(u) \geq \G(v_{\ep})$. Finally, using the concavity of $\G$, we get $\G(u) \ge (1-\ep) \G(v)+\ep \G(u_0)$, and we get \eqref{eq:g5} by letting $\ep$ go to zero. \\

\noindent 
\textbf{Step 4. } Back to the non-normalized equation. 
We have $v_{j}=u_{j}+\mathcal{L}_{j}(u_{j})$ and as shown
above in \eqref{eq:g}, $\mathcal{G}_{j}(u_{j})$ is a bounded sequence (more precisely,
it converges to the maximal value $S$ of $\mathcal{G}$) and $0\leq-\mathcal{E}(u_{j})\leq C,$
which implies that $\mathcal{L}_{j}(u_{j})$ is also a bounded sequence.
After passing to a subsequence we may thus assume that $\mathcal{L}_{j}(u_{j})\rightarrow l\in\R$
so that $v_{j}\rightarrow v:=u+l,$ solving the desired equation (and
$v\in\mathcal{E}^{1}(X,\omega)).$ By the uniqueness of solutions
of the latter equation this means that the whole sequence $u_{j}$
converges to $v$, which concludes the proof. 
\end{proof}

Let us now give the proof of the following result which was essential for step 2: 

\begin{lemm}
\label{lem:energie}
Let $[\omega_{j}]$ and $[\omega]$ be semi-positive big classes such
that $\omega_{j}\rightarrow\omega$ in the $\mathcal{C}^{\infty}-$topology of smooth (semipositive) forms. If $u_{j}\in\mathcal{E}(X,\omega_{j})$
(and $u\in\mathcal{E}(X,\omega))$ such that $u_{j}\rightarrow u$
in $L^{1}(X),$ then 
\[\mathcal{E}_{\omega}(u)\geq\limsup_{j}\mathcal{E}_{\omega_{j}}(u_{j})\]
\end{lemm}

\begin{proof}
When $\omega_{j}=\omega$ the lemma amounts to the well-known fact
that $\mathcal{E}_{\omega}$ is usc. We may as well assume that $\sup u_{j}=\sup u=0.$

\noindent
First of all, we modify the sequence $(u_j)$ to make it non-increasing. More precisely, we set $\tilde u_j := (\sup_{k\ge j}u_k)^*$, which defines an $\om_j$-psh function. Then $\tilde u_j\ge u_j$ and the sequence $(\tilde{u}_j)_j$ is non-increasing. Given $v$ an $\om$-psh function and $c\in \R$, we will write $v^c:=\max \{v,c\}$. By construction $\tilde u_j^c$ decreases to $u^c$, and all these functions are $\om_0$-psh. By the local convergence result of Bedford-Taylor for mixed Monge-Ampère expressions and the smooth convergence of $\om_j$ to $\om$, we see that 
\[\E_{\omega}(u^c) = \lim_{j\to+ \infty}\Ej(\tilde u_j^c)\]
As $u_j \le \tilde u_j \le \tilde u_j^c$, the monotonicity of $\E_{\omega}$ ensures that 
\[\E_{\omega}(u^c) = \limsup_{j\to+ \infty}\Ej(u_j)\]
Taking the infimum over all $c$ and using the definition of the functional $\E_{\omega}$, we obtain the desired inequality. 
\end{proof}

\begin{coro}
Let $(X,D)$ be a log smooth log canonical pair (in particular, the
coefficients of $D$ are in $]-\infty,1])$ and assume that $L:=K_{X}+D$
is semi-positive and big. Fixing an ample line bundle $A$ set $L_{j}:=L+A/j.$
Let $\psi_{j}$ be a decreasing sequence of smooth metrics on on the
klt part $D_{klt}$ of $D$ such that $\psi_{j}\rightarrow\phi_{klt}$
(where $dd^{c}\phi_{klt}=[D_{klt}])$ and consider the Monge-Ampère
equations 
\[(dd^{c}\phi_{j})^{n}=e^{\phi_{j}-\psi_{j}-\phi_{D}}\]
for a metric $\phi_{j}\in\mathcal{E}(X,L_{j}).$ The equations admit
unique solutions $\phi_{j}$ and moreover $\phi_{j}\rightarrow\phi_{KE}$
in $L^{1}$ where $\phi_{KE}$ is the unique solution of equation
\eqref{eq:k-e eq for weights}.
\end{coro}

\section{Smoothness of the Kähler-Einstein metric}
\label{sec:smooth}

Before we go into the details of the proof of the regularity theorem, we would like to give an overview of previous related results and underline the main differences that 
will appear in our specific case, namely the case of general log canonical pairs. As we will rely on the so called logarithmic case (i.e. $(X,D)$ is log smooth, $D$ is reduced, and $K_X+D$ is ample), 
the next section will be devoted to recall some of the main tools appearing in this setting. Then, we will give the proof of the main regularity theorem, which will constitute the core of this section.

\subsection{Special features in the log canonical case}

We should first mention the case of varieties with log terminal singularities, or more generally klt pairs, which correspond to the pairs where the discrepancies $a_i$ defined earlier satisfy the strict inequalities $a_i>-1$. 
Then the situation is relatively well understood: In the non-positively curved case, we know that the Kähler-Einstein metric exists, is unique, has bounded potential, and induces on the regular locus a genuine Kähler-Einstein metric (see e.g. \cite{ BEGZ, EGZ1, EGZ, DemPal}). As for the case of positive curvature, or log Fano manifolds, then there exist criteria (like the properness of the Mabuchi functional) to guarantee the existence and uniqueness (modulo automorphisms of $X$) of a Kähler-Einstein metric \cite{BBEGZ}; this metric is also known to have bounded potential and to be smooth on the regular locus of $X$ (see again \cite{BBEGZ, Paun}). \\

However, the behavior of the Kähler-Einstein metric near the singularities of $X$ is mostly unknown (except if the singularities are orbifold). In the case of a klt pair, we know that the metric will not be smooth along the divisor, but its singularities can sometimes be understood outside of the singular part of $(X,D)$. For example, a recent result in this direction states that the Kähler-Einstein metric has cone singularities near each point where $(X,D)$ is log smooth, i.e. $X$ is smooth and $D$ has simple normal crossing support (cf \cite{G2, GP}).\\

When $(X,D)$ is now a log smooth pair, the situation gets easier because there is no more loss of positivity coming from the resolution of singularities. For example, if the coefficients of $D$ are in $[0,1)$ (the pair is then klt), the Kähler-Einstein metric is known to have cone singularities along $D$, as it was proved by \cite{GP} in full generality, and by \cite{Brendle, CGP, JMR} under some assumptions on $D$. 

When now every coefficient of $D$ is equal to $1$, and $K_X+D$ is ample, then we know from the work of Kobayashi \cite{KobR} and Tian-Yau \cite{Tia} that there exists a unique complete Kähler-Einstein metric having Poincaré singularities along $D$. The situation where the coefficients of $D$ are in $]0,1]$ behaves like a product of cone and Poincaré geometries and was studied in \cite{G12, GP}. 

In a slightly different direction, Tsuji \cite{Tsuji} has considered the case of a singular variety with ample canonical line bundle such that only one divisor appears in its resolution, with discrepancy equal to -1. Finally, Wu \cite{Wu, Wu2} has worked out the case of a quasi-projective manifold compactified by a snc divisor $\sum D_i$ such that $K_X+\sum a_i D_i$ is ample for some coefficients $a_i\ge -1$. 
In our case however, such a strong positivity assumption will never happen as soon as we have to perform a non-trivial resolution. \\

As one can already observe in the log smooth case studied by Kobayashi and Tian-Yau, the log canonical case is very different from the klt case. Let us mention some striking divergences: first of all, the potentials are no more bounded even in the ample case so that the solution does not have minimal singularities. Moreover, the Kähler-Einstein equation in this setting is closely related to a negative curvature geometry. Indeed, if we first consider the Ricci-flat case, then it is impossible to write the equation on the whole $X$. Indeed, the current solution obtained on $\xreg$ will not have finite mass near the singularities, and hence it will not extend as a global positive current on $X$. This phenomenon already happens in \cite{TY}. Finally, it has been proven in \cite[Proposition 3.8]{BBEGZ} that any pair $(X,D)$ with $X$ normal and $-(K_X+D)$ ample admitting a Kähler metric $\om$ on $\xreg$ with continuous potentials solution of $\Ric \om = \om + [D]$ is necessarily klt. Therefore it is pointless to look for positively curved Kähler-Einstein metric in the general setting of log canonical pairs instead of klt pairs. \\

To finish this discussion, let us stress the fact that the class of varieties with semi-log canonical singularities can be realized as a subclass of log canonical pairs (cf definition \ref{def:slc}). This is the largest "reasonable" class to look for Kähler-Einstein metrics: for example, if $X$ is a smooth Fano manifold carrying a smooth divisor $D\in |-K_X|$, then for any $\ep>0$, one has $K_X+(1+\ep)D>0$; however, there is no smooth Kähler-Einstein metric with negative scalar curvature on $X\setminus D$. Indeed, its existence would contradict the Yau-Schwarz lemma applied with the complete Ricci-flat Kähler metric constructed in \cite{TY}. 

Moreover, we will see that the existence of a negatively curved Kähler-Einstein metric on the regular part of a normal projective variety with maximal volume forces the singularities to be at worst log canonical, cf. Proposition \ref{prop:kelc}.

\subsection{The logarithmic case}
\label{sec:log2}
In this section, we will briefly recall the Theorem of Kobayashi and Tian-Yau constructing negatively-curved Kähler-Einstein metrics on quasi-projective varieties $X\setminus D$ where $D$ is a reduced divisor with simple normal crossings, and $K_X+D$ is ample. In the course of the proof of Theorem \ref{thm:an}, we will use in an essential manner the functional spaces introduced by these authors, namely the "quasi-coordinates" version of the usual Hölder spaces $\mathscr C^{k,\al}$. For now, $X_0$ will denote $X\backslash D$.

\begin{defi}
\phantomsection
\label{def:cg2}
We say that a Kähler metric $\om$ on $X_0$ is of Carlson-Griffiths type if there exists a Kähler form $\om_0$ on $X$ such that 
$\om=\om_0-\sum_K \ddc \log \log \frac 1{|s_k|^{2}}$.
\end{defi}

In \cite{CG}, Carlson and Griffiths introduced such a metric for some $\om_0 \in c_1(K_X+\D)$, but one can easily see that such a metric always exists on a Kähler manifold without assumptions on the bundle $K_X+D$. One can also observe that the existence of such a metric $\om$ forces the cohomology class $\{\om\}$ to be Kähler by Demailly's regularization theorem \cite{D1, D2}. 

The reason why we exhibit this particular class of Kähler metrics on $X_0$ having Poincaré singularities along $\D$ is that we have an exact knowledge on its behaviour along $\D$, which is much more precise that knowing its membership in the aforementioned class. This is precisely the class in which one will look for a Kähler-Einstein metric, so that one needs to define the appropriate analogue of the usual Hölder spaces $\mathscr C^{k,\al}$. And to do so, one may (almost) boil down to the usual euclidian situation.\\

The key point is that $(X_0,\om)$ has bounded geometry at any order. Let us get a bit more into the details. To simplify the notations, we will assume that $\D$ is irreducible, so that locally near a point of $\D$, $X_0$ is biholomorphic to $\mathbb{D}^* \times \mathbb{D}^{n-1}$, where $\mathbb{D}$ (resp. $\mathbb{D}^*$) is the unit disc (resp. punctured disc) of $\mathbb C$. We want to show that, roughly speaking, the components of $\om$ in some appropriate coordinates have bounded derivatives at any order. The right way to formalize it consists in introducing quasi-coordinates: they are maps from an open subset $V\subset \mathbb C^n$ to $X_0$ having maximal rank everywhere. So they are just locally invertible, but these maps are not injective in general. \\
To construct such quasi-coordinates on $X_0$, we start from the univeral covering map $\pi:\mathbb{D}\to\mathbb{D}^*$, given by $\pi(w)=e^{\frac{w+1}{w-1}}$. Formally, it sends $1$ to $0$. The idea is to restrict $\pi$ to some fixed ball $B(0,R)$ with $1/2<R<1$, and compose it (at the source) with a biholomorphism $\Phi_{\eta}$ of $\mathbb{D}$ sending $0$ to $\eta$, where $\eta$ is a real parameter which we will take close to $1$. If one wants to write a formula, we set $\Phi_{\eta}(w)=\frac{w+\eta}{1+\eta w}$, so that the quasi-coordinate maps are given by
$\Psi_{\eta}=\pi\circ \Phi_{\eta}\times \mathrm{Id}_{\mathbb{D}^{n-1}}:V=B(0,R)\times \mathbb{D}^{n-1}\to \mathbb{D}^*$, with $\Psi_{\eta}(v,v_2, \ldots, v_n)=(e^{\frac{1+\eta}{1-\eta} \frac{v+1}{v-1}}, v_2, \ldots, v_n)$.\\
Once we have said this, it is easy to see that $X_0$ is covered by the images $\Psi_{\eta}(V)$ when $\eta$ goes to $1$, and for all the trivializing charts for $X$, which are in finite number. Now, an easy computation shows that the derivatives of the components of $\om$ with respect to the $v_i$'s are bounded uniformly in $\eta$. This can be thought as a consequence of the fact that the Poincaré metric is invariant by any biholomorphism of the disc.\\

At this point, it is natural to introduce the Hölder space of $\cka$-functions on $X_0$ using the previously introduced quasi-coordinates:
\begin{defi}
\phantomsection
\label{def:cka2}
For a non-negative integer $k$, a real number $\alpha \in ]0.1[$, we define: 
\[\cka(X_0)=\{u\in \mathscr C^k (X_0);\,\, \sup_{V, \eta} ||u\circ \Psi_{\eta} ||_{k,\alpha}<+\infty\}\]
where the supremum is taken over all our quasi-coordinate maps $V$ (which cover $X_0$). Here $||\cdot ||_{k,\alpha}$ denotes the standard $\cka$-norm for functions defined on a open subset of $\mathbb C^n$.
\end{defi}

The following fact, though easy, is very important (see e.g \cite{KobR} or \cite[Lemma 1.6]{G12} for a detailed proof)  :
\begin{lemm}
\phantomsection
\label{lem:cka2}
Let $\om$ be a Carlson-Griffiths type metric on $X_0$, and $\om_0$ some Kähler metric on $X$. Then 
\[F_0:=\log \left(\prod |s_k|^2 \log^2 |s_k|^2\cdot  \om^n /\om_0^n \right)\]
belongs to the space $\cka(X_0)$ for every $k$ and  $\alpha$.
\end{lemm}

Finding the Kähler-Einstein metric consists then in showing that the Monge-Ampère equation $(\om+\ddc \vp)^n = e^{\vp+f}\om^n$ has a unique solution $\vp\in \cka(X_0)$ for all functions $f\in \cka(X_0)$ with $k\ge 3$. This can be done using the continuity method in the quasi-coordinates. In particular, applying this to $f=F:= -\log \left( \prod |s_k|^2 \log^2 |s_k|^2 \cdot \om^n/\om_0^n \right)+(\mathrm{smooth \, terms \, on\,} X)$, which the previous lemma allows to do, this will prove the existence of a negatively curved Kähler-Einstein metric, which is equivalent to $\om$ (in the \emph{strong} sense: $\vp\in \cka(X_0)$ for all $k,\alpha$).\\

In this continuity method, one needs to obtain first uniform estimates; they follow from a consequence of Yau's maximum principle for complete manifolds which we recall here (see \cite[Proposition 4.1]{CY}):

\begin{prop}
\phantomsection
\label{prop:pm}
Let $(X,\om)$ be a $n$-dimensional complete Kähler manifold, and $F\in \mathscr C^{2}(X)$ bounded from above. We assume that we are given $u\in \mathscr C^{2}(X)$ satisfying $\om+\ddc u>0$ and
\begin{equation}
\label{eqn:cy2}
(\omega+\ddc u)^{n} = e^{u+F}\omega^{n} \nonumber
\end{equation}
Suppose that the bisectional curvature of $(X,\om)$ is bounded below by some constant, and that $u$ is bounded from below. Then 
\[\inf_X u \ge -\sup_X F \quad \textrm{and} \quad \sup_X u \le -\inf_X F \]
\end{prop}

There are similar results for the Laplacian estimates, but as we will not use them directly, we do not state them here. To summarize the discussion, one obtains:

\begin{theo}[Kobayashi \cite{KobR}, Tian-Yau \cite{Tia}]
\phantomsection
\label{thm:kob2}
Let $X$ be a compact Kähler manifold, $D$ a reduced divisor with simple normal crossings, $\om$ a Kähler form of Carlson-Griffiths type on $X\backslash D$, and $F\in \cka(X\backslash D)$ for some $k\ge 3$. Then there exists $\vp \in \cka(X \backslash D)$ solution to the following equation: 
\[(\om+\ddc \vp)^n = e^{\vp+F}\om^n\]
In particular if $K_X+\D$ is ample, then there exists a (unique) Kähler-Einstein metric of curvature $-1$ equivalent to $\om$.
\end{theo}

\subsection{Statement of the regularity theorem}

In this section, we prove that the Kähler-Einstein metric attached to a log canonical pair $(X,D)$ (satisfying $K_X+D$ ample) by Theorem \ref{thm:exist} is smooth on $X_0=\xreg \setminus \Supp(D)$. As usual, we will work on a log resolution $\pi : X'\to X$, where: 
\[K_{X'}=\pi^*(K_X+D)+\sum a_i E_i\]
$E_i$ being either an exceptional divisor or a component of the strict transform of $D$, and the coefficients $a_i$ (called discrepencies) satisfy the inequality $a_i\ge -1$. \\

The Kähler-Einstein metric is given on $X'$ by a (singular) psh weight $\phi$ on $\pi^*(K_X+D)$ satisfying
\[(\ddc \phi)^n = e^{\phi+\sum a_i \phi_{E_i}} \]
where $\phi_{E_i}$ is a psh weight on $\mathcal O_{X'}(E_i)$ such that $\ddc \phi_{E_i} = [E_i]$. So if in local coordinates, $E_i$ is given by $\{z_n=0\}$, then $\phi_{E_i}=\log |z_n|^2$. \\

Our aim is to obtain regularity properties for the solutions of degenerate Monge-Ampère equations like the previous one; this is the content of the following theorem:

\begin{theo}
\label{thm:an}
Let $X$ be a compact Kähler manifold of dimension $n$, $dV$ some volume form, $D=\sum a_i D_i$ a $\R$-divisor with coefficients in $(-\infty, 1]$ and defining sections $s_i$, $E=\sum c_{\al} E_{\al}$ an effective $\R$-divisor such that $D_{red}+E$ has snc support, and $\theta$ a semipositive form with $\int_X\theta^n >0$ such that $\{\theta\} -c_1(E)$ is a Kähler class. Then the $\theta$-psh function $\vp$ with full Monge-Ampère mass, which is a solution of 
\[\la(\theta+\ddc \phi)^n\ra = \frac{e^{\vp}dV}{\prod_{i}|s_i|^{2a_j}} \]
is smooth outside of $\Supp(D) \cup \Supp(E)$.
\end{theo}

Note that although $\vp$ has full Monge-Ampère mass, it is in general far from having minimal singularities as soon as some coefficient $a_i$ of $D$ equals $1$. Think for example of the logarithmic case (a log smooth pair $(X,D)$ where $K_X+D$ is ample; then the potential of the Kähler-Einstein metric is not bounded whereas the class is ample.\\

Let us go back to the general Kähler-Einstein case. We would like to apply the previous results with $E$ being some positive combination of the $E_i$'s. The problem is that there might be no such divisors; for example if $\pi$ happens to be a small resolution, its exceptional locus has codimension at least $2$. Therefore we need to perform another modification. \\

On $X'$, $\pi^* (K_X+D)$ is no more ample, and by \cite[Proposition 1.5]{BBP}, its augmented base locus is $\bp(\pi^* (K_X+D))=\pi^{-1}(\bp(K_X+D)) \cup \exc(\pi)=\exc(\pi)$, and lies above $\xsing \cup \Supp(D)$. It is well-known that one can find a log resolution $\mu:X''\to X'$ of $(X', \bp(\pi^* (K_X+D)))$, and an effective $\Q$-divisor $F$ with snc support lying above $\bp(\pi^*( K_X+D))$ and such that $\mu^*\pi^* (K_X+D)-F$ is ample. Moreover one can also assume that $F+\sum E_i'$ has snc support, where $E_i'$ denotes the strict transform of $E_i$ by $\mu$. \\

Let us recall the argument. We start by resolving the singularities of a Kähler current $T\ge \om$ ($\om$ a Kähler form on $X'$) in $\pi^* (K_X+D)$ computing $\bp(\pi^* (K_X+D))$, then we write Siu's decomposition $\mu^*T=\theta + [D]$ with $\theta$ semi-positive dominating $\mu^* \om $, and $D$ lying above $\bp(\pi^* (K_X+D))$. Finally, we choose a $\mu$-exceptional $\Q$-divisor $G$ such that $-G$ is $\mu$-ample; it exists because $\mu$ is a finite composition of blow-ups with \textit{smooth} centers. Then it becomes clear that for $\ep>0$ small enough, $\{\mu^*\theta \}-\ep G$ is a Kähler class, and we have $\mu^* \pi^* (K_X+D) = (\{\mu^*\theta \}-\ep G)+(\ep G+D)$, with $\ep G+D$ lying above $\bp(\pi^* (K_X+D))$ and having simple normal crossing support. If one had chosen a log resolution of the ideal sheaf generated by the augmented base locus of $\pi^* (K_X+D)$ and the $\mathcal O_X'(E_i)$, we would have obtained the refined result that $F+E'$ has snc support.\\

Set $\nu :=\pi \circ \mu : X'' \to X$, and write $K_{X''}=\nu^* (K_X+D)+E_{\nu}$. We know that $E_{\nu}$ is a divisor with snc support and coefficients $\ge -1$, and by the construction above, there exists a snc divisor $F$ on $X''$ lying above $\xsing \cup \Supp(D)$ such that $F+(E_{\nu})_{\rm red}$ has snc support and $\nu^* (K_X+D)-F$ is ample. Applying Theorem \ref{thm:an}, we get:

\begin{coro}
Let $(X,D)$ be a log canonical pair such that $K_X+D$ is ample. Then the Kähler-Einstein metric $\omke$ on $(X, D)$ is smooth on $\xreg\setminus \Supp(D)$.
\end{coro}

As we shall see in the course of the proof (cf \S \ref{est0}), we do not obtain very precise estimates on the potential of the solution, even at order zero. However, it is tempting to believe that the potential $\phi_{\rm KE}$ of the Kähler-Einstein metric should be locally bounded outside of the non-klt locus of $(X,D)$ defined as the support of the sheaf $\Ox/\mathscr{I}(X,D)$ where $\mathscr{I}(X,D)$ is the multiplier ideal sheaf of $(X,D)$ (cf. e.g. \cite{Laz2}). However, as this locus cannot be read easily on some log resolution, it does not seem obvious how one should tackle this question.

\subsection{Preliminaries: the regularized equation}

We now borrow the notations of Theorem \ref{thm:an}, and we let $\om_0$ be a Kähler form on $X$; it will be our reference metric in the following. 
Recall that we want to solve the equation \[\MA(\vp) = \frac{e^{\vp}dV}{\prod_i |s_i|^{2a_i}}\]
where the unknown function is $\vp$ a $\theta$-psh function, $s_i$ are non-zero sections of $\Ox(D_i)$, $|\cdot |_i$ are smooth hermitian metrics on $\Ox(D_i)$, $f\in \mathscr C ^{\infty}(X)$ and $dV$ is a smooth volume form on $X$. Moreover, the expression $\MA(\vp)$ has to be understood as the non-pluripolar Monge-ampère operator. It will be convenient for the following to differentiate the “klt part“ of $D$ from its “lc part“, so we introduce the following notation:
\[D=\underbrace{\sum_{a_j<1} a_j D_j}_{\Dklt} + \underbrace{\sum_{a_k=1} D_k}_{\Dlc}\]

By Theorem \ref{thm:pert}, we know that the solution is the limit of any sequence of solutions of some appropriate regularized equations. The regularization process we are going to use concerns both the \textit{a priori} non-Kähler class $\{\theta \}$ and the "klt part" in the volume form: $\prod_{a_j<1} |s_i|^{-2a_j}$. More concretely, 
we will be studying the following equation:
\begin{equation}
\label{eq:reg2}
\la(\theta+t\om_0+ \ddc \vp_{t,\ep})^n\ra = \frac{e^{\vp_{t,\ep}+f}dV}{\prod_{a_j<1} (|s_i|^2+\ep ^2)^{a_i} \prod_{a_k=1} |s_k|^2} 
\end{equation}

\subsubsection*{Smoothness of the regularized solution}
At this point, it is still not completely clear that the solution $\vp_{t,\ep}$ of equation \eqref{eq:reg2} is smooth on $X\setminus\Dlc$. So we translate our equation into the logarithmic setting : we set 
\[\om_{t,lc}:=\theta+t\om_0-\sum_{a_k=1} \ddc \log (\log |s_k|^2)^2\]
We may choose the hermitian metrics $|\cdot|_k$ such that $|s_k|<1$ and such that $\om_{t,lc}$ defines a Kähler metric on $X\setminus \Dlc$ (cf \cite{CG, Gri} e.g.). Of course this rescaling will depend on $t$, but we will explain how to bypass this problem later. \\

So using this new reference metric, one may rewrite equation \eqref{eq:reg2} in the following form: 
\[(\omtlc +\ddc \psi_{t,\ep})^n = \frac{e^{\psi_{t,\ep}+f_t}\omtlc^n}{\prod_{a_j<1} (|s_j|^2+\ep ^2)^{a_j}} \]
where $\psi_{t,\ep}=\vpte+\sum_{a_k=1}  \log (\log |s_k|^2)^2$ and $f_t=-\log \left( \frac{\prod_k |s_k|^2 \log ^2 |s_k|^2 \omtlc^n }{dV} \right)$. Clearly, $f_t$ is bounded (but only the lower bound is uniform in $t$) and smooth on $X\setminus \Dlc$, but we know by Lemma \ref{lem:cka2} that $f_t$ is smooth when read in the quasi-coordinates adapted to the pair $(X,\Dlc)$. Therefore, using the Theorem of Kobayashi and Tian-Yau (see Theorem \ref{thm:kob2}), we know that the solution $\pste$ is bounded on $X\setminus \Dlc$: there exists $C_{t,\ep}>0$ such that
\begin{equation}
\label{eq1}
-C_{t,\ep}-\sum_{a_k=1}  \log (\log |s_k|^2)^2 \le \vpte\le C_{t,\ep}-\sum_{a_k=1}  \log (\log |s_k|^2)^2
\end{equation}

Moreover, $\pste$ is smooth in the quasi-coordinates. In particular, $\omtlc +\ddc \psi_{t,\ep}$ is a Kähler metric with bounded geometry on $X\setminus \Dlc$ and with Poincaré type growth along $\Dlc$. Therefore it is complete and has a bounded curvature tensor. To prove the regularity theorem, we will thus have to obtain on each compact subset of $X_0$ estimates on the potential $\vpe$ at any order. \\

\subsubsection*{A first attempt at the uniform estimate}
The previous observation allows us to apply Yau's maximum principle (cf Proposition \ref{prop:pm}), and obtain that \[\sup_{X\setminus \Dlc} \pste \le \sup_{X\setminus \Dlc} \left(\sum a_i \log (|s_i|^2+\ep^2) -f_t\right)\] and similarly $\inf_{X\setminus \Dlc} \pste \le \inf_{X\setminus \Dlc} \left(\sum a_i \log (|s_i|^2+\ep^2) -f_t\right)$. If some coefficient $a_i$ is negative, then we cannot obtain a bound for $\sup \pste$. As for the lower bound, $-f_t$ is not uniformly bounded from below because $\omtlc$ degenerates at $t=0$ (and if some $a_i$ is positive, $a_i \log (|s_i|^2+\ep^2)$ is not uniformly bounded below neither), so we cannot expect to find a lower bound for $\pste$ using this strategy. Therefore we need another method to obtain a zero-order estimate on the potential of the solution. In fact, we will need to add some barrier function to gain positivity, in the spirit of Tsuji's trick \cite{Tsuji88} for the Laplacian estimate of a degenerate class; the novelty in our situation is that this is also needed for the zero-order estimates (as opposed to the klt case). 

\subsection{Uniform estimate}
Before going any further, let us fix some notations.
\subsubsection{A new framework}
 We will denote by $s_i$, $i\in I$ (non-zero) sections of the (reduced) components of the divisor $D_{red}+E$, and by $s_{\alpha}$, $\alpha\in A$ (non-zero) sections of the (reduced) components of $E$; we endow all these line bundles with suitable smooth hermitian metrics (se below). Finally, we set $X_0:=X\setminus (\Supp(D)\cup \Supp(E))$, and define $F:=(D_{red}+E)_{red}$ as the reduced divisor $X\setminus X_0$.

The idea is to work on $X_0$. Of course, if we endow the last space with the Kähler metric $\omtlc$, it will not be complete (near $\Dklt$ e.g.), so we won't be able to use Yau's maximum principle. Instead, we will rather use the following metric: 
\[\om_{\chi} := \theta+t\om_0-\sum_{i \in I} \ddc \log (\log |s_{i}|^2)^2 + \ddc \chi\]
where $\chi:=\sum_{\alpha} c_{\alpha} \log |s_{\alpha}|^2$ (recall that the $c_{\al}$'s are the coefficients of $E$). \\

We do here a slight abuse of notation because $\omc$ depends on $t$. However, the following lemma shows that the dependence is harmless:

\begin{lemm}
\phantomsection
\label{lem1}
Up to changing the previously chosen hermitian metrics, the $(1,1)$-form $\omc$ defines on $X_0$ a smooth Kähler metric with Poincaré growth along $F$ having bounded geometry, all of those properties being satisfied uniformly in $t$.
\end{lemm}
  
What we mean by this statement is that there exists a Poincaré-type metric $\om_P$ on $X_0$ and a constant $C>0$ independent of $t$ such that $C^{-1}\om_P\le \omc \le C \om_P$, and that in the appropriate quasi-coordinates attached to the pair $(X,F)$, the coefficients $g_{i \bar j}$ of $\omc$ satisfy $\left| \frac{\d^{|\al|+|\beta|} g_{i\bar j} } {\d {z}^{\alpha}\d \bar {z}^{\beta} } \right| \le C_{\al, \beta}$ for some constants $C_{\al, \beta}>0$ independant of $t$. In particular, $\omc$ has a uniformly (in $t$) bounded curvature tensor.

\begin{proof}[Proof of Lemma \ref{lem1}]
We know that $\{\theta\}-c_1(E)$ is ample. Therefore, up to changing the hermitian metrics $h_{\alpha}$ on $\Ox(E_{\al})$, we may suppose that $\eta:= \theta - \sum c_{\alpha} \Theta_{h_{\al}}(E_{\al})$
defines a smooth Kähler form on $X$ (we designated by $\Theta_{h_{\al}}(E_{\al})$ the curvature form of the hermitian line bundle $(E_{\al}, h_{\al}$)). Therefore, on $X_0$, we have: 
\[\omc = \eta + t\om_0-\sum_{i \in I} \ddc \log (\log |s_{i}|^2)^2\]
and the statement follows easily from the computations of \cite[Proposition 2.1]{CG} and \cite[Lemma 2]{KobR} or \cite{Tia}.
\end{proof}

\subsubsection{Getting the lower bound}
\label{est0}
First of all, we will use the crucial information that $\vpte$ converges (in the weak sense of distributions) to some $\theta$-psh function (cf first paragraph). By the elementary properties of psh (or quasi-psh) functions, we know that $(\vpte)$ is uniformly bounded above on the compact set $X$ (see e.g. \cite[Theorem 3.2.13]{Hor2}). Therefore, we obtain some uniform constant $C$ such that 
\begin{equation}
\label{eq2}
\vpte\le C
\end{equation}

Now, recall that we chose $\omc$ to be the new reference metric, so our equation becomes 
\begin{equation}
\label{eq3}
(\omc +\ddc \ute)^n = e^{\ute+G_{\ep}} \omc^n 
\end{equation}
where \[\ute:= \vpte +\sum_{i\in I}  \log (\log |s_i|^2)^2-\chi\] and 
\[G_{\ep}:=\chi+f+\sum_{a_j<1}\log \left( \frac{|s_j|^2}{(|s_j|^2+\ep^2)^{a_j}}\right)-\log\left(\frac{\prod_{i\in I} |s_i|^2 \log^2 |s_i|^2 \omc ^n}{dV}\right)\]

Here again we should mention that $G_{\ep}$ also depends on $t$ through the last term involving $\omc$. For our purpose, we can ignore this dependence in order to simplify the notations. 

We can see from \eqref{eq2} and Lemma \ref{lem1} that $G_{\ep}$ has a uniform (in $t$ and $\ep$) upper bound on $X_0$:
\[\sup_{X_0} G_{\ep} \le C\] 
Moreover, we know from \eqref{eq1} that $\vpte +\sum_{a_k=1}  \log (\log |s_k|^2)^2$ is bounded. Therefore, it follows immediately that $\ute$ is bounded \textit{from below} (but a priori non uniformly). Applying Yau's maximum principle (cf. Proposition \ref{prop:pm}) to the smooth function $\ute$ on the complete Kähler manifold $(X_0, \omc)$ ensures that $\inf_{X_0}\ute \ge - \sup_{X_0} G_{\ep} \ge -C$. In terms of $\vpte$, and recalling inequality \eqref{eq2} we get: 
\[\boxed{ C \, \ge \,  \vpte \, \ge\, \chi - C- \sum_{i\in I} \log (\log |s_i|^2)^2}\]

\subsection{Laplacian estimate}

For the Laplacian estimate, we still work on the open manifold $X_0$. We endow it with the complete Kähler metric $\om_{\chi}$, and we recall from Lemma \ref{lem1} that $\omc$ has uniformly bounded (bisectional) curvature. \\

As usual when one wants to compare to Kähler metrics $\om$ and $\om'$, the strategy is to use an inequality of the form $\Delta F \ge G$, where $F,G$ involve terms like $\tr_{\om} \om'$, $\tr_{\om'} \om$ or the local potentials of $\om'-\om$. There exist several variants of such inequalities, due e.g. to Chen-Lu, Yau, Siu, etc. involving different assumptions on the curvature of the metrics involved. In our case, as we have a control on the bisectional curvature of the reference metric $\om_{\chi}$, on the Ricci curvature of the "unknown metric" $\omc+\ddc u_{t,\ep}$, and on the laplacian $\Delta_{\omc} G_{\ep}$, we could use any of these formulas. \\

We have chosen to use a variant of Siu's inequality \cite[p.99]{Siu}, which can be found in \cite[Proposition 2.1]{CGP} (see also \cite{Paun, BBEGZ}); notice the important feature allowing the factor $e^{-F_-}$ for $F_-$ quasi-psh which is crucial for us since the RHS of our Monge-Ampère equation has poles:

\begin{prop}
\phantomsection
\label{prop:est2}
Let $X$ be a Kähler manifold of dimension $n$, $\om, \om'$ two cohomologous Kähler metrics on $X$. We assume that $\om'=\om+\ddc u$ with $\om'^n=e^{u+F_+-F_-} \om^n$. and that we have a constant $C>0$ satisfying:
\begin{enumerate}
\item[$(i)$] $\ddc F_{\pm} \ge -C \om$, 
\item[$(ii)$] $\Theta_{\om}(T_X) \ge -C \om \otimes \mathrm{Id}_{T_X}$.
\end{enumerate}
Then there exist some constant $A>0$ depending only on $n$ and $C$ such that 
\[\Delta_{\om'} (\log \tr_{\om} \om' - A u+F_-) \ge \tr_{\om'} \om -nA.\]
Moreover, if one assumes that $\sup F_+ \le C, u \ge -C$ and that $\log \tr_{\om} \om' - A u +F_-$ attains its maximum on $X$, then there exists $M>0$ depending on $n$ and $C$ only such that: 
\[ \om' \le M e^{Au-F_-} \om.\]
\end{prop} 

Here $\Delta$ (resp. $\Delta'$) is the laplacian with respect to $\om$ (resp. $\om'$), and $\Theta_{\om}(T_X)$ is the Chern curvature tensor of the hermitian holomorphic vector bundle $(T_X,\omega)$ (which
may be identified with the tensor of holomorphic bisectional curvatures, usually denoted by the letter R). 

\begin{proof}[Sketch of proof]
Siu's inequality applied to $\om=\sum g_{i\bar j} dz_i\wedge d\bar z_k$ and $\om'=\sum g_{i\bar j}' dz_i\wedge d\bar z_k$ yields:
\begin{equation*}
\Delta' (\log\tr_{\om} \om' ) \ge
\frac{1}{\tr_{\om}\om'}\big(-g^{j\bar i}
R_{j\bar i}+ 
 \Delta (u+F_+ -F_-)+ g^{\prime k\bar l} R^{j\bar i}_{ k\bar l}g^\prime_{j\bar i}\big)
\end{equation*}

Recollecting terms coming (with different signs) from the scalar and the Ricci curvature, we will obtain a similar inequality involving only a lower bound for the holomorphic bisectional curvature, namely

\begin{equation}
\label{eq4}
\Delta' \log \tr_{\om}\om'\ge \, \frac{\Delta (u+F_+ -F_-)}{\tr_{\om}(\om')}-B \,  \tr_{\om'}\om
\end{equation}
where $B$ is a lower bound for the bisectional curvature of $\om$: this is the content of \cite[Lemma 2.2]{CGP}.\\
Clearly, $\Delta u = \tr_{\om} \om'-n $ so that $\Delta(u+F_+) \ge -n(C+1)$. As $ \tr_{\om} \om'  \tr_{\om'} \om \ge n$, we get
\begin{equation}
\label{eq5}
\frac{\Delta(u+F_+)}{\tr_{\om}{\om'}} \ge -(1+C) \tr_{\om'} \om 
\end{equation}
As for the second laplacian, we write
\[0\le C\om+\ddc F_- \le \tr_{\om'}(C\om+\ddc F_-)\om'\]
and we take the trace with respect to $\om$:
\[\frac{nC+\Delta_{\om}F_-}{\tr_{\om}\om'} \le C \tr_{\om'}\om+\Delta_{\om'}F_-\]
so that 
\begin{equation}
\label{eq6}
\Delta_{\om'}F_- \ge \frac{\Delta_{\om}F_-}{\tr_{\om}\om'}-C \tr_{\om'}\om
\end{equation}
Plugging \eqref{eq5} and \eqref{eq6} into \eqref{eq4}, we get:
\[\Delta' (\log \tr_{\om} \om' +F_-)\ge -C_1 \tr_{\om'} \om\]
for $C_1=1+B+2C$. 
Finally, using $\Delta' u =n-\tr_{\om'}\om$, we see that 
\[\Delta' (\log \tr_{\om} \om' - (C_1+1) u+F_-) \ge \tr_{\om'} \om -n(C_1+1)\]
which shows the first assertion by chosing $A:=1+C_1$.\\

As for the second part, if we denote by $p$ the point where the maximum is attained, then one has $(\tr_{\om'}\om)(p) \le C_2$. Using the basic inequality $\tr_{\om}\om' \le e^{u+F_+-F_-} (\tr_{\om'}\om)^{n-1}$, one gets
\begin{eqnarray*}
\log(\tr_{\om}\om')& = & (\log \tr_{\om}\om' - Au +F_-) + A u-F_-  \\
&\le & (u(p)+F_+(p)-F_-(p))+(n-1) \log (nA)-Au(p)+F_-(p) +Au-F_-\\
&\le &C_2+Au-F_-
\end{eqnarray*}
where $C_2=\sup F_++(n-1)\log (nA)-(A-1) \inf u$ (recall that $A$ can be chosen to be positive).
This concludes the proof of the proposition.
\end{proof}

Recall that we are interested in equation \eqref{eq3} given by
\[(\omc +\ddc \ute)^n = e^{\ute+G_{\ep}} \omc^n \]
We obtained the zero-order estimate on $\ute$ in the last section, and now we want a Laplacian estimate. In order too use the previous proposition we first have to decompose $G_{\ep}$ as a difference of $C\omc$-psh functions in order to use the result above. Recall that \[G_{\ep}:=\chi+f+\sum_{a_j<1}\log \left( \frac{|s_j|^2}{(|s_j|^2+\ep^2)^{a_j}}\right)-\log\left(\frac{\prod_{i\in I} |s_i|^2 \log^2 |s_i|^2 \omc ^n}{dV}\right)\]
By \cite{KobR} or \cite[Lemma 1.6]{G12}, the last term is already known to be smooth in the quasi-coordinates (and it depends neither on $t$ nor on $\ep$).\\
We claim that 
\[G_{\ep}=\underbrace{\left[\chi+f+\sum_{ a_j<1} \log |s_j|^2+\sum_{a_j<0}\log(|s_j|^2+\ep^2)^{-a_j}\right]}_{F_+}-\underbrace{\left[ \sum_{0<a_j<1} \log(|s_j|^2+\ep^2)^{a_j}\right]}_{F_-}\]
gives the desired decomposition $G_{\ep}=F_+-F_-$ in the notations of Proposition \ref{prop:est2}. Indeed, $\chi, f, \log |s_i|^2$ are quasi-psh, thus $C\omc$-psh for some uniform $C>0$ as $\omc $ dominates some fixed Kähler form, cf Lemma \ref{lem1}. Moreover, a simple computation leads to the identity:
\[\ddc \log (|s|^2+\ep^2) = \frac{\ep^2}{(|s|^2+\ep^2)^2}\cdotp \la D's, D's \ra -\frac{|s|^2}{|s|^2+\ep^2}\cdotp \Theta_h  \]
where $\Theta_h$ is the curvature of the hermitian metric $h$ implicit in the term $|s|^2$, and $D's$ is the $(1,0)$-part of $Ds$ where $D$ is the Chern connection attached to $(\Ox(\mathrm{div}(s)), h)$. If follows that $F_{\pm}$ are $C\omc$-psh for some uniform $C>0$. \\

We can now apply Proposition \ref{prop:est2} to the setting: $\om=\omc, \om'= \omc+\ddc \ute, F_+-F_-=G_{\ep}$. Indeed, it is clear that $F_+$ is uniformly upper bounded, we just saw that $F_{\pm}$ are $C\omc$-psh, and we know from the previous section that $\ute$ has a uniform lower bound. Furthermore, $\log \tr_{\om} \om' - A u+F_-$ attains its maximum on $X_0$: indeed, $-Au$ tends to $-\infty$ near the boundary of $X_0$, $F_-$ is bounded (it is even smooth), and $\tr_{\om}\om' = \Delta_{\omc}(\vpte+\sum_{i\in I}  \log (\log |s_i|^2)^2-\chi)$ is bounded on $X_0$ (we know it for the term $\Delta_{\omc}(\vpte-\chi)$ and it is an elementary computation for the other term). \\

In conclusion, we may use Proposition \ref{prop:est2} to obtain the following estimate: 
\[\boxed{\theta + t\om_0 + \ddc \vpte \le M \left(\prod_{i\in I} \log^2 |s_i|^2\right)^C \cdot \prod_{\alpha \in A} |s_{\alpha}|^{-c_{\alpha}\cdot C}\prod_{a_j>0}|s_j|^{-2a_j}\,  \omc}\]

For the "reverse inequality", we use the identity
\[(\omc +\ddc \ute)^n = e^{\ute+G_{\ep}} \omc^n \]
which leads to the inequality 
\[\tr_{\omc + \ddc \ute}\omc \le e^{-(\ute+G_{\ep})} \left(\tr_{\omc}(\omc+\ddc \ute) \right)^{n-1}\]
and therefore
\[\theta + t\om_0 + \ddc \vpte \ge M^{-1} \prod_{a_j<1} \frac{|s_j|^2}{(|s_j|^2+\ep^2)^{a_j}} \cdot \left(\prod_{i\in I} \log^2 |s_i|^2\right)^{-C} \cdot \prod_{\alpha \in A} |s_{\alpha}|^{c_{\alpha}\cdot C} \prod_{a_j>0}|s_j|^{2a_j}\,  \omc\]
for some uniform $C,M>0$ (different from the previous ones). \\

In particular, for any compact set $K \Subset X_0$, there exists a constant $C_K>0$ satisfying
\[C_K^{-1} \, \om_0 \le \theta + t\om_0 + \ddc \vpte \le C_K \, \om_0\]
Using Evans-Krylov theorem and the classical elliptic theory shows that the potential $\vpte$ satisfies uniform $\mathscr C^{k,\al}$ estimates on any $\Omega \Subset K$ for each $k,\al$. Thus the theorem is proved.

\begin{rema}
One can easily obtain somewhat more precise estimates. Indeed, if $\alpha$ is  a nef and big class and $E$ an effective $\R$-divisor such that $\al-E$ is ample, then we have in fact that for every $\delta>0$, 
$\al-\delta E$ is ample (write $\al-\delta E = (1-\delta )\al+\delta (\al-E)$). Applying this observation to Theorem \ref{thm:an}, we see that for every $\delta>0$, there exists $C_{\delta}>0$ such that: 
\[\vp_{\mathrm {KE}} \, \ge\,  \delta \sum c_{\al} \log |s_{\al}|^2-\sum_{i\in I} \log (\log |s_i|^2)^2-C_{\delta}\]
for every $\delta >0$. One could apply the same argument to the Laplacian estimates.
\end{rema}

\subsubsection*{About uniqueness of the Kähler-Einstein metrics}

First of all, in the case where $X$ is smooth and $K_X$ is ample, then uniqueness of the Kähler-Einstein metric constructed by Aubin and Yau
is a straightforward consequence of the maximum principle. 
Generalizing this principle to some complete Kähler manifolds in \cite{Yau}, Yau could prove that on a Kähler manifold, there can be only one \textit{complete}
Kähler metric $\om$ satisfying $\Ric \om= -\om$. In particular, this result has been applied by Kobayashi and Tian-Yau
to show the uniqueness of the Kähler-Einstein metric for a log smooth pair $(X,D)$ satisfying $K_X+D$ ample, cf \cite{KobR, Tia}. Using Yau-Schwarz lemma in a more subtle way (through the notion of
almost complete metric), they also show uniqueness when $K_X+D$ is only assumed nef, big and ample modulo $D$, which means that $K_X+D$ intersects positively every curve not contained in $D$. For example, if $(X,D)$ is a log resolution
of some canonically polarized singular variety, these assumptions are not satisfied. \\

In our situation we proceed in a different manner: we first use the volume assumption (as a replacement for completeness) to show that the Kähler-Einstein metric, originally defined on the (log) regular locus , extends to define a a positive current on X whose local potentials glue to define a solution with full Monge-Amère mass of a global Monge-Ampère equation to which we can apply the comparison principle to finally deduce the uniqueness.

\section{Applications}
\label{sec:app}
\subsection{Yau-Tian-Donaldson conjecture for singular varieties}
\label{sec:ytd}

Let us start with the following converse of Theorem A stated in the
introduction.

\begin{prop}
\phantomsection
\label{prop:kelc}
Let $X$ be projective variety satisfying the conditions $G_1$ and $S_2$, and such that $K_{X}$ is $\Q$-ample. If $X$ admits
a Kähler-Einstein metric $\omega$ in the following sense: $\omega$
is a K-E metric on $\xreg$ and its total volume there is equal
to $K_{X}^{n},$ then $X$ has semi-log canonical singularities. 
\end{prop}

\begin{proof}
We begin with the case where $X$ is normal. 

Let us first show that $\phi:=\log\omega^{n}$ on $\xreg$ extends
to an element in $\mathcal{E}(X,K_{X}).$ By the K-E equation $dd^{c}\phi:=-\mbox{Ric \ensuremath{\omega}}=\omega\geq0$
and hence, since $X$ is normal, $\phi$ extends to a positively curved
singular metric on $K_{X}$ over all of $X.$ Thus, writing $\omega=dd^{c}\psi$
for some other such metric $\psi$ the compactness of $X$ forces
the K-E equation $\MA(\phi)=e^{\phi}$ (up to shifting $\phi$ by a
constant) globally on $X$ (since the non-pluripolar MA does not charge
the singular locus of $X).$ Moreover, by the volume assumption we
have that $\int_{X}\MA(\phi)=K_{X}^{n}$ and hence $\mathcal{\phi\in E}(X,K_{X}),$
as desired.

Next, fix a resolution $\pi:\, X'\rightarrow X$ and assume, to a get
a contradiction, that $p^{*}K_{X}=K_{X'}+D$ where $D$ is a snc $\Q-$divisor
such that $D=D'+(1+\delta)E$ for some $\delta>0,$ where $D'$ is
a $\Q-$divisor and $E$ is a smooth irreducible divisor transversal
to the support of $D'.$ Since $\phi$ has maximal MA-mass it follows,
as shown in \cite{BBEGZ} using an Izumi type estimate, that $\pi^{*}\phi$ has
no Lelong numbers. In particular, it follows from the characterization of Lelong numbers that there exists a neigbourhood $U$
of $E$ such that $\pi^{*}\phi\geq\frac{1}{2}\delta\log|s_{E}|^{2}-C$
in local trivializations. Moreover, we may take $U$ such that $D$
does not intersect $U.$ But then it follows from the K-E equation
that 
\[K_{X}^{n}\geq C'\int_{U}e^{\pi^{*}\phi-(1+\delta)\log|s_{E}|^{2}}\geq C''\int_{U}e^{-(1+\delta/2)\log|s_{E}|^{2}}=\infty,\]
 which gives the desired contradiction.\\
 
 We move on to the general case when $X$ is only assumed to be $G_1$ and $S_2$. As we observed in \S \ref{sec:sing}, the result of Proposition \ref{prop:ext} holds actually in 
 the general $G_1$ and $S_2$ case (we did not use at all that the singularities were slc); however we should be careful and work instead on a log-resolution of $(\Xnu, C_{\Xnu})$ because the formula $\nu^* K_X= K_{\Xnu}+C_{\Xnu}$ could not be meaningful anymore if $C_{\Xnu}$ is not Cartier, cf \S \ref{sec:cond} and the remarks following the identity \eqref{eq:can}. So the first conclusion is that  the weight $\phi:=\log\omega^{n}$ on $\xreg$ extends on the normalization $\Xnu$ to a psh weight in $\mathcal{E}(\Xnu,\nu^*K_{X}).$ Then we take a log-resolution $\pi:X' \to \Xnu$ of the pair $(\Xnu, C_{\Xnu})$ where $C_{\Xnu}$ is the conductor of the normalization; a priori, this is just an effective divisor, possibly non-reduced. We write $(X',D')$ for the new pair that we obtain on $X'$. Then same arguments as earlier show that $D'$ has coefficients less than or equal to $1$, which amounts to saying that $(\Xnu, C_{\Xnu})$ is log canonical, or equivalently that $X$ has semi-log canonical singularities. 
\end{proof}

To relate this to $K$-stability we recall that Odaka \cite{Odaka} has
shown that, if $X$ is $K$-semistable, then $X$ has semi-log canonical singularities
(recall that we assume that $X$ is $G_1$ and $S_2$ and that $K_{X}>0$).
Conversely, if $X$ is semi-log canonical, then $X$ is $K$-stable \cite{od2}. 
Hence, combining our results with Odaka's results gives the following
confirmation of the Yau-Tian-Donaldson conjecture for varieties being $G_1$ and $S_2$,
with $K_{X}$ ample:
\begin{theo}
Let $X$ be a $G_1$ and $S_2$ projective variety such that $K_{X}$ is ample.
Then $X$ admits a Kähler-Einstein metric iff $X$ is $K$-stable.
\end{theo}

It would be interesting to have a direct analytical proof of the implication
``Kähler-Einstein implies $K$-stable'' as shown in \cite{Ber} the (log) Fano case (where $K$-stability has to be replaced by $K$-polystability in the
presence of holomorphic vector fields).\\

\subsection{Automorphism groups of canonically polarized varieties}

The existence and uniqueness of Kähler-Einstein metrics established
in Theorem A allows us to give an analytical proof of the following
result shown in \cite{BHPS} (where two proofs were given, one
cohomological and one geometric)

\begin{theo}
Let $X$ be a stable variety. Then $\mathrm{Aut}(X)$ is finite. 
\end{theo}

\begin{proof}
First observe that the normalization map $\nu: X^{\nu}\to X$ induces a morphism $\mathrm{Aut}(X) \to \mathrm{Aut}(X^{\nu})$ which is clearly injective. So we only need to treat the case of a normal stable variety $X$. As $X$ is canonically polarized and every automorphism of $X$ preserves $K_X$, the automorphism group of $X$ can be realized the automorphism group of a polarized variety, hence it is an algebraic subgroup of $\mathrm{PGL}(N, \mathbb C)$ for some $N$, and therefore it has finitely many connected components. It is thus enough to show that this group has dimension $0$. 

\medskip

By general results on automorphism groups of normal varieties (see
\cite[Lemma 5.2]{BBEGZ} and references therein) it is equivalent to show that any
holomorphic vector field $V$ on $\xreg$ vanishes identically.
To prove this vanishing we first observe that, by normality, $V$
is the infinitisimal generator of a complex one-parameter family of
automorphism $F$ of $X$ and in particular of $\xreg$. Fix a Kähler-Einstein
current $\omega$ on $X.$ By the naturality of the KE-equation it
follows that $F^{*}\omega$ is also a KE-current and hence by uniqueness
$F^{*}\omega=\omega$ on $\xreg$. Let us denote by $V_{r}$ and
$V_{i}$ the real and imagnary parts of $V,$ which are infinitisimal
generators of real one parameter families of automorphisms that we
will denote by $F_{r}$ and $F_{i}$ respectively, which, by the previous
argument, also preserve $\omega.$ Next, note that any automorphism
automatically lifts to the line bundle $K_{X}$ over $\xreg$ and
thus it follows from general principles that the real part $V_{r}$
of $V$ is a Hamiltonian vector field, i.e. $i_{V_{r}}\omega=dh$
for some smooth function $h$ on $X_{reg}.$ But then, by Cartan's
formula, the Lie deriviative $L_{V_{i}}\omega$ is given by $d(i_{W_{r}}\omega)=dJi_{V_{r}}\omega=dJdh=dd^{c}h.$
Since the flow $F_{i}$ defined by $V_{i}$ also preserves $\omega$
it thus follows that $dd^{c}h=0.$ But by normality it follows that
$h=0$ (indeed, by normality $h$ is bounded and we can thus apply
the maximum principle on a resolution). Hence $i_{V_{r}}\omega=0$
on $X_{reg},$ which forces $V_{r}=0$ on $X_{reg},$ since $\omega$
is Kähler there and in particular pointwise non-degenerate. Finally,
by the same argument $V_{i}=0$ (for example replacing $V$ with $JV)$
and hence $V=0$ as desired. 
\end{proof}

\begin{rema}
In the case when $X$ is smooth there is a simple cohomological proof
of the previous proposition: by Serre duality $H^{0}(X,TX)$ is isomorphic
to $H^{n-1}(X,-K_{X}),$ which is trivial by Kodaira vanishing (since
$K_{X}$ is ample). In the case when $X$ is log canonical a similar
cohomological argument can be used \cite{BHPS}, relying on the
Bogomolov-Sommese vanishing result for log canonical singularities,
established in \cite[Theorem 7.2]{GKKP}. Indeed, if $V$ does not
vanish identically then contracting with $V$ on $\xreg$ maps $K_{X}$
to a rank one reflexive sheaf in $\mbox{Hom }(K_{X},\Omega_{X}^{[n-1]}),$
where $\Omega_{X}^{[n-1]}$ is the sheaf of reflexive $(n-1)-$forms
on $X$ and hence by the Bogomolov-Sommese vanishing result in \cite{GKKP}
the Kodaira dimension of $K_{X}$ is at most $n-1,$ which contradict
the ampleness of $K_{X}.$ 
\end{rema}

\section{Outlook}
\label{sec:outlook}
\subsection{Towards Miyaoka-Yau type inequalities}

For simplicity we will only consider the case $n=2$ (but a similar discussion applies in the general case). We set $E:=\Omega_{X}^{1},$
the cotangent bundle of $X.$ The classical case is when $X$ is smooth
with $K_{X}$ ample, where the Miyaoka-Yau inequality
says
\[
c_{1}(E)^{2}\leq3c_{2}(E).
\]
Let us briefly recall Yau's differential-geometric proof. We equip
$E$ with the Hermitian metric induced by $\omega$ and denote by
$(E,\omega)$ the corresponding Hermitian vector bundle. Then, if
$\omega$ is Kähler-Einstein a direct local calculation gives the
\emph{point-wise} inequality 
\[
c_{1}(E,\omega)^{2}\leq3c_{2}(E,\omega)
\]
formulated in terms of the Chern-Weil representatives $c_{i}(E,\omega)$
of the corresponding Chern classes. Hence, integrating immediately
gives the Miyaoka-Yau inequality. Repeating this argument
in the singular case when $X$ a stable surface and using Theorem
A gives the following
\begin{prop}
The following inequality holds for a stable surface equipped with
the canonical Kähler-Einstein metric $\omega$ on its regular part:
\[
c_1(K_{X})^{2}\leq3\int_{\xreg}c_{2}(E,\omega)
\]
with equality iff $\omega$ has constant holomorphic sectional curvature,
i.e. $(\xreg,\omega)$ is locally isometric to a ball.\end{prop}
\begin{proof}
Since the point-wise inequality above still holds, by the KE-condition,
we can simply integrate it over $X_{reg}$ and use that, by Theorem
A, $c_{1}(K_{X})^{2}=\int_{\xreg}c_{1}(E,\omega)^{2}.$ The conditions
for equality are well-known in the point-wise inequality.
\end{proof}
Since $\omega$ is canonically attached to $X$ one could simply define
the rhs appearing in the inequality above as the ``analytical second
Chern number'' $c_{2,\rm an}(X)$ of $X.$ However, it should be stressed
that it is not even a priori clear that $c_{2,an}(X)$ is finite,
even though we expect that this is the case. More precisely, we expect
that $c_{2,\rm an}(X)$ can be identified with (or at least bounded from
above) by a suitable algebraically defined second Chern class number
$c_{2}(X).$ Various definitions of such Chern numbers have been proposed
in the litterature and we refer the reader to the paper of Langer
\cite{Langer} where very general algebraic Miyaoka-Yau type inequalities
are obtained, which in particular apply to stable surfaces. More generally,
as before, our arguments apply to log canonical pairs.

\subsection{The Weil-Petersson geometry of the moduli space of stable varieties}

In this section we will briefly explain how the finite energy property
of the Kähler-Einstien metric on a stable variety naturally appears
in the geometric study of the moduli space $\mathcal{M}$ of all stable
varieties. In a nut shell, the Kähler-Einstein metrics on stable varieties
induces a metric on the $\mathbb{Q}-$line bundle $\mathcal{L}\rightarrow\mathcal{M}$
over the moduli space defined by the top Deligne pairing of $K_{X}$
and the finite energy condition is precisely the condition which makes
sure that the metric is point-wise finite. The relation to Weil-Petterson geometry
comes from the well-known fact that the curvature form of the corresponding
metric over the moduli space $\mathcal{M}_{0}$ of all\emph{ smooth}
stable varieties (i.e. all canonically polarized $n-$dimensional
manifolds, with a fixed oriented smooth structure) coincides with
the Weil-Petersson metric $\Omega_{WP}$ on $\mathcal{M}_{0}$ \cite{FS,Schum}.

To be a bit more precise we first recall that given a line bundle
$L\rightarrow X$ over a (complex) $n-$dimensional algebraic variety
$X$ its top Deligne pairing i.e. the $(n+1)-$fold Deligne pairing
of $L$ with itself is a complex line that we will denote by $\left\langle L\right\rangle $
\cite{El1, El2}. Equipping  $\left\langle L\right\rangle $ with an Hermitian metric $\phi$ (using
additive notation as before) induces a Hermitian metric $\left\langle \phi\right\rangle $ on  $\left\langle L\right\rangle $, 
satisfying the change of metric formula: $\left\langle \phi\right\rangle -\left\langle \psi\right\rangle =(\mathcal{E}(\phi)-\mathcal{E}(\psi))$
(up to a multiplicative normalization constant), where $\mathcal{E}$
is the energy functional appearing in \S \ref{sec:var} (compare \cite{PS}).
Fixing a smooth reference metric $\phi_{0}$ on $L$ one can use the
latter transformation formula to define the metric $\left\langle \phi\right\rangle $as
long as $\phi$ has finite energy. The resulting metric $\left\langle \phi\right\rangle $
is then independent of the choice of reference metric $\phi_{0}.$
More generally, in the relative case of a flat morphism $\mathcal{X}\rightarrow B$
between integral schemes of relative dimension $n$ and an Hermitian
line bundle $L\rightarrow\mathcal{X}$ this construction produces
an Hermitian line bundle $\left\langle L\right\rangle \rightarrow B$
over the base $B.$

In particular, taking $X$ to be an $n-$dimensional stable variety
$L:=K_{X}$ one obtains a canonical metric on the complex line $\left\langle K_{X}\right\rangle ,$
induced by the finite energy metric on $K_{X}$ determined by the
volume form of the Kähler-Einstein metric on the regular locus of
$X.$ 

Let now  $\mathcal{M}$ denote the moduli space of all n-dimensional stable varieties with a
fixed Hilbert polynomial \cite{koll, Kollar}. Using the existence of a universal stable family  $\mathcal{X}$
(in the sense of Kollar) over a finite cover of each irreducible component of the moduli space 
one obtains a $\Q-$ line bundle $\mathcal{L}$ over $\mathcal{M}$,  induced by the fiber-wise top Deligne pairings $\left\langle K_{X}\right\rangle .$
We conjecture that the metric on $\mathcal{L}$ induced by the fiber-wise
Kähler-Einstein metrics is continuous (as in the case of stable curves
\cite{fr}). Confirming this conjecture would require a more detailed
analysis of the dependence of the Kähler-Einstein metric on the complex
structure that we leave for the future.

As is well-known the curvature of the corresponding metric over \emph{$\mathcal{M}_{0}$}
coincides (up to a numerical factor) with the Weil-Petersson metric
$\Omega_{WP}.$ In particular, it is stricly positive as a form (in
the orbifold sense). Under the validity of the previous conjecture
one thus obtains a canonical extension of the induced Weil-Petersson
metric on \emph{$\mathcal{M}_{0}$} to its compactification in $\mathcal{M}$
as a positive current with continuous potentials. It would also be
very interesting to know under which assumptions the extension is
\emph{strictly} positive in a suitable sense, for example if it is,
locally, the restriction of a Kähler metric? These problems are (e.g. by Grauert's generalization of
Kodaira's embedding theorem to singular varities) closely
related to the problem of showing that the correponding line bundle
$\mathcal{L}$ over the moduli space $\mathcal{M}$ is ample (on each
irreducible component) and it should be compared with the recent work
of Schumacher \cite{Schum}, where an analytic proof of the quasi-projectivity
of $\mathcal{M}_{0}$ is given. As shown by Schumacher the Weil-Petersson
metric $\Omega_{WP}$ on $\mathcal{M}_{0}$ admits a (non-canonical)
extension as a positive current with analytic singularities to Artin's
Moishezon compactifaction of $\mathcal{M}_{0}.$ But the conjecture
above is closely related to the problem of obtaining a \emph{canonical} 
extension of $\Omega_{WP}$ to $\mathcal{M}$ as a positive current
with continuous potentials.

Finally, it should be pointed out that the top Deligne pairing used
above essentially coincides with Tian's CM-line bundle in this setting
(by the Knudson-Mumford expansion and Zhang's isomorphism realizing
the Chow divisor as a top Deligne pairing). The ampleness of the induced
CM-line bundle over general moduli spaces of $K$-stable polarized varieties
was recently speculated on by Odaka \cite{od}.

\bibliographystyle{smfalpha}
\bibliography{biblio2}

\end{document}